\documentclass[12pt]{article}
\usepackage{latexsym,amssymb,upref,amsmath,amsthm, amsfonts,authblk}
\usepackage{color}
\usepackage{graphicx}

\usepackage{fullpage}

\usepackage{float}

\usepackage[math]{cellspace}
\usepackage[pdftex,hypertexnames=false,linktocpage=true]{hyperref}

\usepackage{diagbox}
\usepackage{makecell}
\usepackage{adjustbox}

\usepackage{subcaption}

\usepackage{tikz}
\usetikzlibrary{decorations.pathreplacing}

\newtheorem{thm}{Theorem}[section] 
\newtheorem{lemma}[thm]{Lemma}
\newtheorem{cor}[thm]{Corollary} 
\newtheorem{claim}[thm]{Claim}
\newtheorem{prop}[thm]{Proposition}

\theoremstyle{definition}
\newtheorem{defn}[thm]{Definition}

\newtheorem{conj}[thm]{Conjecture}

\theoremstyle{definition}
\newtheorem*{const*}{Construction}

\newcommand{\N}{\mathcal{N}}

\DeclareMathOperator{\ex}{ex}

\usepackage{graphicx} 
\usepackage{tikz}

\usetikzlibrary{patterns}
\usetikzlibrary {calc}

\definecolor{ao(english)}{rgb}{0.0,0.5,0.0}

\definecolor{szurke}{rgb}{0.7,0.7,0.7}

\newcommand{\BEtwoone}{\vc{
\begin{tikzpicture}
\def\r{1.5} 
        \coordinate (O) at (0,0);
        \coordinate (a) at ($(O)+(0:\r cm)$);
        \coordinate (b) at ($(O)+(180:\r cm)$);
        
        \coordinate (ab) at ($ (a)!.5!(b)$);
        
        \def\w{0.6} 

        \filldraw[color = szurke, fill=szurke!90, thick] 
    ($(a)+(120:1cm)$)
    .. controls ($ (ab)!\w cm!90:(a)$).. ($(b)+(60:1cm)$)
    arc (60:-60:1)
    .. controls ($ (ab)!\w cm!90:(b)$)..  ($(a)+(240:1cm)$)
    arc (240:120:1)
    ;

        \draw[thick]  (a) circle (1cm);
        \draw[thick]  (b) circle (1cm);

\end{tikzpicture}
}}

\newcommand{\BEtwotwo}{\vc{
\begin{tikzpicture}
\def\r{1.5} 
        \coordinate (O) at (0,0);
        \coordinate (a) at ($(O)+(0:\r cm)$);
        \coordinate (b) at ($(O)+(180:\r cm)$);
        
        \coordinate (ab) at ($ (a)!.5!(b)$);

        \def\w{0.6} 

        \filldraw[color=ao(english), fill=ao(english)!70, thick] 
    ($(a)+(120:1cm)$)
    .. controls ($ (ab)!\w cm!90:(a)$).. ($(b)+(60:1cm)$)
    arc (60:-60:1)
    .. controls ($ (ab)!\w cm!90:(b)$)..  ($(a)+(240:1cm)$)
    arc (240:120:1)
    ;

        \draw[thick]  (a) circle (1cm);
        \draw[thick]  (b) circle (1cm);
\end{tikzpicture}
}}

\newcommand{\BEthreeone}{\vc{
\begin{tikzpicture}
\def\r{1.6} 
        \coordinate (O) at (0,0);
        \coordinate (a) at ($(O)+(0:\r cm)$);
        \coordinate (b) at ($(O)+(120:\r cm)$);
        \coordinate (c) at ($(O)+(240:\r cm)$);

        \coordinate (ab) at ($ (a)!.5!(b)$);
        \coordinate (bc) at ($ (b)!.5!(c)$);
        \coordinate (ca) at ($ (c)!.5!(a)$);
    
        \def\w{0.6} 

        \filldraw[color = szurke, fill=szurke!90, thick] 
    ($(a)+(90:1cm)$)
    .. controls ($ (ab)!\w cm!90:(a)$).. ($(b)+(15:1cm)$)
    arc (15:-150:1)
    .. controls ($ (bc)!\w cm!90:(b)$)..  ($(c)+(150:1cm)$)
    arc (150:-15:1)
     .. controls ($ (ca)!\w cm!90:(c)$)..  ($(a)+(255:1cm)$)
     arc (255:90:1)
    ;
    
        \draw[thick]  (a) circle (1cm);
        \draw[thick]  (b) circle (1cm);
        \draw[thick]  (c) circle (1cm);

\end{tikzpicture}
}}

\newcommand{\BEthreetwo}{\vc{
\begin{tikzpicture}
\def\r{1.6} 
        \coordinate (O) at (0,0);
        \coordinate (a) at ($(O)+(0:\r cm)$);
        \coordinate (b) at ($(O)+(120:\r cm)$);
        \coordinate (c) at ($(O)+(240:\r cm)$);

        \coordinate (ab) at ($ (a)!.5!(b)$);
        \coordinate (bc) at ($ (b)!.5!(c)$);
        \coordinate (ca) at ($ (c)!.5!(a)$);
    
        \def\w{0.6} 

        \filldraw[color=ao(english), fill=ao(english)!70, thick] 
    ($(a)+(90:1cm)$)
    .. controls ($ (ab)!\w cm!90:(a)$).. ($(b)+(15:1cm)$)
    arc (15:-150:1)
    .. controls ($ (bc)!\w cm!90:(b)$)..  ($(c)+(150:1cm)$)
    arc (150:-15:1)
     .. controls ($ (ca)!\w cm!90:(c)$)..  ($(a)+(255:1cm)$)
     arc (255:90:1)
    ;

    \filldraw[color = szurke, fill=szurke!90, very thick]
    ($(b)+(-150:1cm)$)
    .. controls ($ (bc)!\w cm!90:(b)$)..  ($(c)+(150:1cm)$)
     arc (150:30:1)
     .. controls ($ (bc)!\w cm!90:(c)$)..  ($(b)+(-30:1cm)$)
     arc (330:210:1)
    ;
    
        \draw[thick]  (a) circle (1cm);
        \draw[thick]  (b) circle (1cm);
        \draw[thick]  (c) circle (1cm);

\end{tikzpicture}
}}

\newcommand{\BEthreethree}{\vc{
\begin{tikzpicture}
\def\r{1.6} 
        \coordinate (O) at (0,0);
        \coordinate (a) at ($(O)+(0:\r cm)$);
        \coordinate (b) at ($(O)+(120:\r cm)$);
        \coordinate (c) at ($(O)+(240:\r cm)$);

        \coordinate (ab) at ($ (a)!.5!(b)$);
        \coordinate (bc) at ($ (b)!.5!(c)$);
        \coordinate (ca) at ($ (c)!.5!(a)$);
    
        \def\w{0.6} 

        \filldraw[color=ao(english), fill=ao(english)!70, thick] 
    ($(a)+(90:1cm)$)
    .. controls ($ (ab)!\w cm!90:(a)$).. ($(b)+(15:1cm)$)
    arc (15:-150:1)
    .. controls ($ (bc)!\w cm!90:(b)$)..  ($(c)+(150:1cm)$)
    arc (150:-15:1)
     .. controls ($ (ca)!\w cm!90:(c)$)..  ($(a)+(255:1cm)$)
     arc (255:90:1)
    ;
    
        \draw[thick]  (a) circle (1cm);
        \draw[thick]  (b) circle (1cm);
        \draw[thick]  (c) circle (1cm);

\end{tikzpicture}
}}

\newcommand{\BEfourone}{\vc{
\begin{tikzpicture}
\def\r{1.9} 
        \coordinate (O) at (0,0);
        \coordinate (a) at ($(O)+(0:\r cm)$);
        \coordinate (b) at ($(O)+(90:\r cm)$);
        \coordinate (c) at ($(O)+(180:\r cm)$);
        \coordinate (d) at ($(O)+(270:\r cm)$);

        \coordinate (ab) at ($ (a)!.5!(b)$);
        \coordinate (bc) at ($ (b)!.5!(c)$);
        \coordinate (cd) at ($ (c)!.5!(d)$);
        \coordinate (da) at ($ (d)!.5!(a)$);
    
        \def\w{0.6} 

        \filldraw[color = szurke, fill=szurke!90, thick] 
    ($(a)+(75:1cm)$)
    .. controls ($ (ab)!\w cm!90:(a)$).. ($(b)+(15:1cm)$)
    arc (15:-205:1)
    .. controls ($ (bc)!\w cm!90:(b)$)..  ($(c)+(105:1cm)$)
    arc (105:-105:1)
     .. controls ($ (cd)!\w cm!90:(c)$)..  ($(d)+(195:1cm)$)
     arc (195:-15:1)
     .. controls ($ (da)!\w cm!90:(d)$)..  ($(a)+(305:1cm)$)
     arc (305:75:1)
    ;

        \draw[thick]  (a) circle (1cm);
        \draw[thick]  (b) circle (1cm);
        \draw[thick]  (c) circle (1cm);
        \draw[thick]  (d) circle (1cm);

\end{tikzpicture}
}}

\newcommand{\BEfourtwo}{\vc{
\begin{tikzpicture}
\def\r{1.9} 
        \coordinate (O) at (0,0);
        \coordinate (a) at ($(O)+(0:\r cm)$);
        \coordinate (b) at ($(O)+(90:\r cm)$);
        \coordinate (c) at ($(O)+(180:\r cm)$);
        \coordinate (d) at ($(O)+(270:\r cm)$);

        \coordinate (ab) at ($ (a)!.5!(b)$);
        \coordinate (bc) at ($ (b)!.5!(c)$);
        \coordinate (cd) at ($ (c)!.5!(d)$);
        \coordinate (da) at ($ (d)!.5!(a)$);
    
        \def\w{0.6} 

        \filldraw[color=ao(english), fill=ao(english)!70, thick] 
    ($(a)+(75:1cm)$)
    .. controls ($ (ab)!\w cm!90:(a)$).. ($(b)+(15:1cm)$)
    arc (15:-205:1)
    .. controls ($ (bc)!\w cm!90:(b)$)..  ($(c)+(105:1cm)$)
    arc (105:-105:1)
     .. controls ($ (cd)!\w cm!90:(c)$)..  ($(d)+(195:1cm)$)
     arc (195:-15:1)
     .. controls ($ (da)!\w cm!90:(d)$)..  ($(a)+(305:1cm)$)
     arc (305:75:1)
    ;

    \filldraw[color = szurke, fill=szurke!90, very thick]
    ($(a)+(75:1cm)$)
    .. controls ($ (ab)!\w cm!90:(a)$)..  ($(b)+(15:1cm)$)
     arc (15:-105:1)
     .. controls ($ (ab)!\w cm!90:(b)$)..  ($(a)+(195:1cm)$)
     arc (195:75:1)
    ;

    \filldraw[color = szurke, fill=szurke!90, very thick]
    ($(c)+(-105:1cm)$)
    .. controls ($ (cd)!\w cm!90:(c)$)..  ($(d)+(195:1cm)$)
     arc (195:75:1)
     .. controls ($ (cd)!\w cm!90:(d)$)..  ($(c)+(15:1cm)$)
     arc (15:-105:1)
    ;

        \draw[thick]  (a) circle (1cm);
        \draw[thick]  (b) circle (1cm);
        \draw[thick]  (c) circle (1cm);
        \draw[thick]  (d) circle (1cm);

\end{tikzpicture}
}}

\newcommand{\BEfourthree}{\vc{
\begin{tikzpicture}
\def\r{1.9} 
        \coordinate (O) at (0,0);
        \coordinate (a) at ($(O)+(0:\r cm)$);
        \coordinate (b) at ($(O)+(90:\r cm)$);
        \coordinate (c) at ($(O)+(180:\r cm)$);
        \coordinate (d) at ($(O)+(270:\r cm)$);

        \coordinate (ab) at ($ (a)!.5!(b)$);
        \coordinate (bc) at ($ (b)!.5!(c)$);
        \coordinate (cd) at ($ (c)!.5!(d)$);
        \coordinate (da) at ($ (d)!.5!(a)$);
    
        \def\w{0.6} 

        \filldraw[color=ao(english), fill=ao(english)!70, thick] 
    ($(a)+(75:1cm)$)
    .. controls ($ (ab)!\w cm!90:(a)$).. ($(b)+(15:1cm)$)
    arc (15:-205:1)
    .. controls ($ (bc)!\w cm!90:(b)$)..  ($(c)+(105:1cm)$)
    arc (105:-105:1)
     .. controls ($ (cd)!\w cm!90:(c)$)..  ($(d)+(195:1cm)$)
     arc (195:-15:1)
     .. controls ($ (da)!\w cm!90:(d)$)..  ($(a)+(305:1cm)$)
     arc (305:75:1)
    ;

    \filldraw[color = szurke, fill=szurke!90, very thick]
    ($(c)+(-105:1cm)$)
    .. controls ($ (cd)!\w cm!90:(c)$)..  ($(d)+(195:1cm)$)
     arc (195:75:1)
     .. controls ($ (cd)!\w cm!90:(d)$)..  ($(c)+(15:1cm)$)
     arc (15:-105:1)
    ;

        \draw[thick]  (a) circle (1cm);
        \draw[thick]  (b) circle (1cm);
        \draw[thick]  (c) circle (1cm);
        \draw[thick]  (d) circle (1cm);

\end{tikzpicture}
}}

\newcommand{\BEfourfour}{\vc{
\begin{tikzpicture}
\def\r{1.9} 
        \coordinate (O) at (0,0);
        \coordinate (a) at ($(O)+(0:\r cm)$);
        \coordinate (b) at ($(O)+(90:\r cm)$);
        \coordinate (c) at ($(O)+(180:\r cm)$);
        \coordinate (d) at ($(O)+(270:\r cm)$);

        \coordinate (ab) at ($ (a)!.5!(b)$);
        \coordinate (bc) at ($ (b)!.5!(c)$);
        \coordinate (cd) at ($ (c)!.5!(d)$);
        \coordinate (da) at ($ (d)!.5!(a)$);
    
        \def\w{0.6} 

        \filldraw[color=ao(english), fill=ao(english)!70, thick] 
    ($(a)+(75:1cm)$)
    .. controls ($ (ab)!\w cm!90:(a)$).. ($(b)+(15:1cm)$)
    arc (15:-205:1)
    .. controls ($ (bc)!\w cm!90:(b)$)..  ($(c)+(105:1cm)$)
    arc (105:-105:1)
     .. controls ($ (cd)!\w cm!90:(c)$)..  ($(d)+(195:1cm)$)
     arc (195:-15:1)
     .. controls ($ (da)!\w cm!90:(d)$)..  ($(a)+(305:1cm)$)
     arc (305:75:1)
    ;

        \draw[thick]  (a) circle (1cm);
        \draw[thick]  (b) circle (1cm);
        \draw[thick]  (c) circle (1cm);
        \draw[thick]  (d) circle (1cm);

\end{tikzpicture}
}}

\newcommand{\BEfiveone}{\vc{
\begin{tikzpicture}
\def\r{2.2} 
        \coordinate (O) at (0,0);
        \coordinate (a) at ($(O)+(0:\r cm)$);
        \coordinate (b) at ($(O)+(72:\r cm)$);
        \coordinate (c) at ($(O)+(144:\r cm)$);
        \coordinate (d) at ($(O)+(216:\r cm)$);
        \coordinate (e) at ($(O)+(288:\r cm)$);

        \coordinate (ab) at ($ (a)!.5!(b)$);
        \coordinate (bc) at ($ (b)!.5!(c)$);
        \coordinate (cd) at ($ (c)!.5!(d)$);
        \coordinate (de) at ($ (d)!.5!(e)$);
        \coordinate (ea) at ($ (e)!.5!(a)$);
    
        \def\w{0.6} 

        \filldraw[color = szurke, fill=szurke!90, thick] 
    ($(a)+(66:1cm)$)
    .. controls ($ (ab)!\w cm!90:(a)$).. ($(b)+(6:1cm)$)
    arc (366:138:1)
    .. controls ($ (bc)!\w cm!90:(b)$)..  ($(c)+(78:1cm)$)
    arc (78:-150:1)
     .. controls ($ (cd)!\w cm!90:(c)$)..  ($(d)+(150:1cm)$)
     arc (150:-78:1)
     .. controls ($ (de)!\w cm!90:(d)$)..  ($(e)+(222:1cm)$)
     arc (222:-6:1)
     .. controls ($ (ea)!\w cm!90:(e)$)..  ($(a)+(294:1cm)$)
     arc (294:66:1)
    ;

        \draw[thick]  (a) circle (1cm);
        \draw[thick]  (b) circle (1cm);
        \draw[thick]  (c) circle (1cm);
        \draw[thick]  (d) circle (1cm);
        \draw[thick]  (e) circle (1cm);
\end{tikzpicture}
}}

\newcommand{\BEfivetwo}{\vc{
\begin{tikzpicture}
\def\r{2.2} 
        \coordinate (O) at (0,0);
        \coordinate (a) at ($(O)+(0:\r cm)$);
        \coordinate (b) at ($(O)+(72:\r cm)$);
        \coordinate (c) at ($(O)+(144:\r cm)$);
        \coordinate (d) at ($(O)+(216:\r cm)$);
        \coordinate (e) at ($(O)+(288:\r cm)$);

        \coordinate (ab) at ($ (a)!.5!(b)$);
        \coordinate (bc) at ($ (b)!.5!(c)$);
        \coordinate (cd) at ($ (c)!.5!(d)$);
        \coordinate (de) at ($ (d)!.5!(e)$);
        \coordinate (ea) at ($ (e)!.5!(a)$);

        \coordinate (be) at ($ (b)!.5!(e)$);
    
        \def\w{0.6} 

        \filldraw[color=ao(english), fill=ao(english)!70, thick] 
    ($(a)+(66:1cm)$)
    .. controls ($ (ab)!\w cm!90:(a)$).. ($(b)+(6:1cm)$)
    arc (366:138:1)
    .. controls ($ (bc)!\w cm!90:(b)$)..  ($(c)+(78:1cm)$)
    arc (78:-150:1)
     .. controls ($ (cd)!\w cm!90:(c)$)..  ($(d)+(150:1cm)$)
     arc (150:-78:1)
     .. controls ($ (de)!\w cm!90:(d)$)..  ($(e)+(222:1cm)$)
     arc (222:-6:1)
     .. controls ($ (ea)!\w cm!90:(e)$)..  ($(a)+(294:1cm)$)
     arc (294:66:1)
    ;

    \filldraw[color = szurke, fill=szurke!90, very thick]
    ($(a)+(66:1cm)$)
    .. controls ($ (ab)!\w cm!90:(a)$).. ($(b)+(6:1cm)$)
    arc (366:210:1)
    .. controls ($ (be)!0.35*\w cm!90:(b)$)..  ($(e)+(150:1cm)$)
    arc (150:-6:1)
    .. controls ($ (ea)!\w cm!90:(e)$)..  ($(a)+(294:1cm)$)
     arc (294:66:1)
    ;

    \filldraw[color = szurke, fill=szurke!90, very thick]
    ($(c)+(-150:1cm)$)
    .. controls ($ (cd)!\w cm!90:(c)$)..  ($(d)+(150:1cm)$)
     arc (150:30:1)
     .. controls ($ (cd)!\w cm!90:(d)$)..  ($(c)+(330:1cm)$)
     arc (330:210:1)
    ;

        \draw[thick]  (a) circle (1cm);
        \draw[thick]  (b) circle (1cm);
        \draw[thick]  (c) circle (1cm);
        \draw[thick]  (d) circle (1cm);
        \draw[thick]  (e) circle (1cm);

\end{tikzpicture}
}}

\newcommand{\BEfivethree}{\vc{
\begin{tikzpicture}
\def\r{2.2} 
        \coordinate (O) at (0,0);
        \coordinate (a) at ($(O)+(0:\r cm)$);
        \coordinate (b) at ($(O)+(72:\r cm)$);
        \coordinate (c) at ($(O)+(144:\r cm)$);
        \coordinate (d) at ($(O)+(216:\r cm)$);
        \coordinate (e) at ($(O)+(288:\r cm)$);

        \coordinate (ab) at ($ (a)!.5!(b)$);
        \coordinate (bc) at ($ (b)!.5!(c)$);
        \coordinate (cd) at ($ (c)!.5!(d)$);
        \coordinate (de) at ($ (d)!.5!(e)$);
        \coordinate (ea) at ($ (e)!.5!(a)$);
    
        \def\w{0.6} 

        \filldraw[color=ao(english), fill=ao(english)!70, thick] 
    ($(a)+(66:1cm)$)
    .. controls ($ (ab)!\w cm!90:(a)$).. ($(b)+(6:1cm)$)
    arc (366:138:1)
    .. controls ($ (bc)!\w cm!90:(b)$)..  ($(c)+(78:1cm)$)
    arc (78:-150:1)
     .. controls ($ (cd)!\w cm!90:(c)$)..  ($(d)+(150:1cm)$)
     arc (150:-78:1)
     .. controls ($ (de)!\w cm!90:(d)$)..  ($(e)+(222:1cm)$)
     arc (222:-6:1)
     .. controls ($ (ea)!\w cm!90:(e)$)..  ($(a)+(294:1cm)$)
     arc (294:66:1)
    ;

    \filldraw[color = szurke, fill=szurke!90, very thick]
    ($(e)+(-6:1cm)$)
    .. controls ($ (ea)!\w cm!90:(e)$)..  ($(a)+(294:1cm)$)
     arc (294:174:1)
     .. controls ($ (ea)!\w cm!90:(a)$)..  ($(e)+(114:1cm)$)
     arc (114:-6:1)
    ;

    \filldraw[color = szurke, fill=szurke!90, very thick]
    ($(c)+(-150:1cm)$)
    .. controls ($ (cd)!\w cm!90:(c)$)..  ($(d)+(150:1cm)$)
     arc (150:30:1)
     .. controls ($ (cd)!\w cm!90:(d)$)..  ($(c)+(330:1cm)$)
     arc (330:210:1)
    ;

        \draw[thick]  (a) circle (1cm);
        \draw[thick]  (b) circle (1cm);
        \draw[thick]  (c) circle (1cm);
        \draw[thick]  (d) circle (1cm);
        \draw[thick]  (e) circle (1cm);

\end{tikzpicture}
}}

\newcommand{\BEfivefour}{\vc{
\begin{tikzpicture}
\def\r{2.2} 
        \coordinate (O) at (0,0);
        \coordinate (a) at ($(O)+(0:\r cm)$);
        \coordinate (b) at ($(O)+(72:\r cm)$);
        \coordinate (c) at ($(O)+(144:\r cm)$);
        \coordinate (d) at ($(O)+(216:\r cm)$);
        \coordinate (e) at ($(O)+(288:\r cm)$);

        \coordinate (ab) at ($ (a)!.5!(b)$);
        \coordinate (bc) at ($ (b)!.5!(c)$);
        \coordinate (cd) at ($ (c)!.5!(d)$);
        \coordinate (de) at ($ (d)!.5!(e)$);
        \coordinate (ea) at ($ (e)!.5!(a)$);
    
        \def\w{0.6} 

        \filldraw[color=ao(english), fill=ao(english)!70, thick] 
    ($(a)+(66:1cm)$)
    .. controls ($ (ab)!\w cm!90:(a)$).. ($(b)+(6:1cm)$)
    arc (366:138:1)
    .. controls ($ (bc)!\w cm!90:(b)$)..  ($(c)+(78:1cm)$)
    arc (78:-150:1)
     .. controls ($ (cd)!\w cm!90:(c)$)..  ($(d)+(150:1cm)$)
     arc (150:-78:1)
     .. controls ($ (de)!\w cm!90:(d)$)..  ($(e)+(222:1cm)$)
     arc (222:-6:1)
     .. controls ($ (ea)!\w cm!90:(e)$)..  ($(a)+(294:1cm)$)
     arc (294:66:1)
    ;

    \filldraw[color = szurke, fill=szurke!90, very thick]
    ($(c)+(-150:1cm)$)
    .. controls ($ (cd)!\w cm!90:(c)$)..  ($(d)+(150:1cm)$)
     arc (150:30:1)
     .. controls ($ (cd)!\w cm!90:(d)$)..  ($(c)+(330:1cm)$)
     arc (330:210:1)
    ;

        \draw[thick]  (a) circle (1cm);
        \draw[thick]  (b) circle (1cm);
        \draw[thick]  (c) circle (1cm);
        \draw[thick]  (d) circle (1cm);
        \draw[thick]  (e) circle (1cm);

\end{tikzpicture}
}}

\newcommand{\BEfivefive}{\vc{
\begin{tikzpicture}
\def\r{2.2} 
        \coordinate (O) at (0,0);
        \coordinate (a) at ($(O)+(0:\r cm)$);
        \coordinate (b) at ($(O)+(72:\r cm)$);
        \coordinate (c) at ($(O)+(144:\r cm)$);
        \coordinate (d) at ($(O)+(216:\r cm)$);
        \coordinate (e) at ($(O)+(288:\r cm)$);

        \coordinate (ab) at ($ (a)!.5!(b)$);
        \coordinate (bc) at ($ (b)!.5!(c)$);
        \coordinate (cd) at ($ (c)!.5!(d)$);
        \coordinate (de) at ($ (d)!.5!(e)$);
        \coordinate (ea) at ($ (e)!.5!(a)$);
    
        \def\w{0.6} 

        \filldraw[color=ao(english), fill=ao(english)!70, thick] 
    ($(a)+(66:1cm)$)
    .. controls ($ (ab)!\w cm!90:(a)$).. ($(b)+(6:1cm)$)
    arc (366:138:1)
    .. controls ($ (bc)!\w cm!90:(b)$)..  ($(c)+(78:1cm)$)
    arc (78:-150:1)
     .. controls ($ (cd)!\w cm!90:(c)$)..  ($(d)+(150:1cm)$)
     arc (150:-78:1)
     .. controls ($ (de)!\w cm!90:(d)$)..  ($(e)+(222:1cm)$)
     arc (222:-6:1)
     .. controls ($ (ea)!\w cm!90:(e)$)..  ($(a)+(294:1cm)$)
     arc (294:66:1)
    ;

        \draw[thick]  (a) circle (1cm);
        \draw[thick]  (b) circle (1cm);
        \draw[thick]  (c) circle (1cm);
        \draw[thick]  (d) circle (1cm);
        \draw[thick]  (e) circle (1cm);

\end{tikzpicture}
}}

\newcommand{\BEsixthree}{\vc{
\begin{tikzpicture}
\def\r{2.5} 
        \coordinate (O) at (0,0);
        \coordinate (a) at ($(O)+(0:\r cm)$);
        \coordinate (b) at ($(O)+(60:\r cm)$);
        \coordinate (c) at ($(O)+(120:\r cm)$);
        \coordinate (d) at ($(O)+(180:\r cm)$);
        \coordinate (e) at ($(O)+(240:\r cm)$);
        \coordinate (f) at ($(O)+(300:\r cm)$);

        \coordinate (ab) at ($ (a)!.5!(b)$);
        \coordinate (bc) at ($ (b)!.5!(c)$);
        \coordinate (cd) at ($ (c)!.5!(d)$);
        \coordinate (de) at ($ (d)!.5!(e)$);
        \coordinate (ef) at ($ (e)!.5!(f)$);
        \coordinate (fa) at ($ (f)!.5!(a)$);

        \def\w{0.6} 

        \filldraw[color=ao(english), fill=ao(english)!70, thick] 
    ($(a)+(60:1cm)$)
    .. controls ($ (ab)!\w cm!90:(a)$).. ($(b)+(0:1cm)$)
    arc (360:120:1)
    .. controls ($ (bc)!\w cm!90:(b)$)..  ($(c)+(60:1cm)$)
    arc (60:-180:1)
     .. controls ($ (cd)!\w cm!90:(c)$)..  ($(d)+(120:1cm)$)
     arc (120:-120:1)
     .. controls ($ (de)!\w cm!90:(d)$)..  ($(e)+(180:1cm)$)
     arc (180:-60:1)
     .. controls ($ (ef)!\w cm!90:(e)$)..  ($(f)+(240:1cm)$)
     arc (240:0:1)
     .. controls ($ (fa)!\w cm!90:(f)$)..  ($(a)+(300:1cm)$)
     arc (300:60:1)
    ;

    \filldraw[color = szurke, fill=szurke!90, very thick]
    ($(c)+(60:1cm)$)
    .. controls ($ (bc)!\w cm!90:(b)$)..  ($(b)+(120:1cm)$)
    arc (120:240:1)
    .. controls ($ (bc)!\w cm!90:(c)$)..  ($(c)+(300:1cm)$)
    arc (-60:60:1)
    ;

    \filldraw[color = szurke, fill=szurke!90, very thick]
    ($(f)+(0:1cm)$)
    .. controls ($ (fa)!\w cm!90:(f)$)..  ($(a)+(300:1cm)$)
    arc (300:180:1)
    .. controls ($ (fa)!\w cm!90:(a)$)..  ($(f)+(120:1cm)$)
    arc (120:0:1)
    ;

    \filldraw[color = szurke, fill=szurke!90, very thick]
    ($(e)+(60:1cm)$)
    .. controls ($ (de)!\w cm!90:(e)$)..  ($(d)+(0:1cm)$)
    arc (360:240:1)
    .. controls ($ (de)!\w cm!90:(d)$)..  ($(e)+(180:1cm)$)
    arc (180:60:1)
    ;

        \draw[thick]  (a) circle (1cm);
        \draw[thick]  (b) circle (1cm);
        \draw[thick]  (c) circle (1cm);
        \draw[thick]  (d) circle (1cm);
        \draw[thick]  (e) circle (1cm);
        \draw[thick]  (f) circle (1cm);

\end{tikzpicture}
}}

\newcommand{\BEsixfour}{\vc{
\begin{tikzpicture}
\def\r{2.5} 
        \coordinate (O) at (0,0);
        \coordinate (a) at ($(O)+(0:\r cm)$);
        \coordinate (b) at ($(O)+(60:\r cm)$);
        \coordinate (c) at ($(O)+(120:\r cm)$);
        \coordinate (d) at ($(O)+(180:\r cm)$);
        \coordinate (e) at ($(O)+(240:\r cm)$);
        \coordinate (f) at ($(O)+(300:\r cm)$);

        \coordinate (ab) at ($ (a)!.5!(b)$);
        \coordinate (bc) at ($ (b)!.5!(c)$);
        \coordinate (cd) at ($ (c)!.5!(d)$);
        \coordinate (de) at ($ (d)!.5!(e)$);
        \coordinate (ef) at ($ (e)!.5!(f)$);
        \coordinate (fa) at ($ (f)!.5!(a)$);

        \def\w{0.6} 

        \filldraw[color=ao(english), fill=ao(english)!70, thick] 
    ($(a)+(60:1cm)$)
    .. controls ($ (ab)!\w cm!90:(a)$).. ($(b)+(0:1cm)$)
    arc (360:120:1)
    .. controls ($ (bc)!\w cm!90:(b)$)..  ($(c)+(60:1cm)$)
    arc (60:-180:1)
     .. controls ($ (cd)!\w cm!90:(c)$)..  ($(d)+(120:1cm)$)
     arc (120:-120:1)
     .. controls ($ (de)!\w cm!90:(d)$)..  ($(e)+(180:1cm)$)
     arc (180:-60:1)
     .. controls ($ (ef)!\w cm!90:(e)$)..  ($(f)+(240:1cm)$)
     arc (240:0:1)
     .. controls ($ (fa)!\w cm!90:(f)$)..  ($(a)+(300:1cm)$)
     arc (300:60:1)
    ;

    \filldraw[color = szurke, fill=szurke!90, very thick]
    ($(f)+(0:1cm)$)
    .. controls ($ (fa)!\w cm!90:(f)$)..  ($(a)+(300:1cm)$)
    arc (300:180:1)
    .. controls ($ (fa)!\w cm!90:(a)$)..  ($(f)+(120:1cm)$)
    arc (120:0:1)
    ;

    \filldraw[color = szurke, fill=szurke!90, very thick]
    ($(e)+(60:1cm)$)
    .. controls ($ (de)!\w cm!90:(e)$)..  ($(d)+(0:1cm)$)
    arc (360:240:1)
    .. controls ($ (de)!\w cm!90:(d)$)..  ($(e)+(180:1cm)$)
    arc (180:60:1)
    ;

        \draw[thick]  (a) circle (1cm);
        \draw[thick]  (b) circle (1cm);
        \draw[thick]  (c) circle (1cm);
        \draw[thick]  (d) circle (1cm);
        \draw[thick]  (e) circle (1cm);
        \draw[thick]  (f) circle (1cm);

\end{tikzpicture}
}}

\newcommand{\BEsixfive}{\vc{
\begin{tikzpicture}
\def\r{2.5} 
        \coordinate (O) at (0,0);
        \coordinate (a) at ($(O)+(0:\r cm)$);
        \coordinate (b) at ($(O)+(60:\r cm)$);
        \coordinate (c) at ($(O)+(120:\r cm)$);
        \coordinate (d) at ($(O)+(180:\r cm)$);
        \coordinate (e) at ($(O)+(240:\r cm)$);
        \coordinate (f) at ($(O)+(300:\r cm)$);

        \coordinate (ab) at ($ (a)!.5!(b)$);
        \coordinate (bc) at ($ (b)!.5!(c)$);
        \coordinate (cd) at ($ (c)!.5!(d)$);
        \coordinate (de) at ($ (d)!.5!(e)$);
        \coordinate (ef) at ($ (e)!.5!(f)$);
        \coordinate (fa) at ($ (f)!.5!(a)$);

        \def\w{0.6} 

        \filldraw[color=ao(english), fill=ao(english)!70, thick] 
    ($(a)+(60:1cm)$)
    .. controls ($ (ab)!\w cm!90:(a)$).. ($(b)+(0:1cm)$)
    arc (360:120:1)
    .. controls ($ (bc)!\w cm!90:(b)$)..  ($(c)+(60:1cm)$)
    arc (60:-180:1)
     .. controls ($ (cd)!\w cm!90:(c)$)..  ($(d)+(120:1cm)$)
     arc (120:-120:1)
     .. controls ($ (de)!\w cm!90:(d)$)..  ($(e)+(180:1cm)$)
     arc (180:-60:1)
     .. controls ($ (ef)!\w cm!90:(e)$)..  ($(f)+(240:1cm)$)
     arc (240:0:1)
     .. controls ($ (fa)!\w cm!90:(f)$)..  ($(a)+(300:1cm)$)
     arc (300:60:1)
    ;

    \filldraw[color = szurke, fill=szurke!90, very thick]
    ($(e)+(60:1cm)$)
    .. controls ($ (de)!\w cm!90:(e)$)..  ($(d)+(0:1cm)$)
    arc (360:240:1)
    .. controls ($ (de)!\w cm!90:(d)$)..  ($(e)+(180:1cm)$)
    arc (180:60:1)
    ;

        \draw[thick]  (a) circle (1cm);
        \draw[thick]  (b) circle (1cm);
        \draw[thick]  (c) circle (1cm);
        \draw[thick]  (d) circle (1cm);
        \draw[thick]  (e) circle (1cm);
        \draw[thick]  (f) circle (1cm);

\end{tikzpicture}
}}

\newcommand{\BEsixsix}{\vc{
\begin{tikzpicture}
\def\r{2.5} 
        \coordinate (O) at (0,0);
        \coordinate (a) at ($(O)+(0:\r cm)$);
        \coordinate (b) at ($(O)+(60:\r cm)$);
        \coordinate (c) at ($(O)+(120:\r cm)$);
        \coordinate (d) at ($(O)+(180:\r cm)$);
        \coordinate (e) at ($(O)+(240:\r cm)$);
        \coordinate (f) at ($(O)+(300:\r cm)$);

        \coordinate (ab) at ($ (a)!.5!(b)$);
        \coordinate (bc) at ($ (b)!.5!(c)$);
        \coordinate (cd) at ($ (c)!.5!(d)$);
        \coordinate (de) at ($ (d)!.5!(e)$);
        \coordinate (ef) at ($ (e)!.5!(f)$);
        \coordinate (fa) at ($ (f)!.5!(a)$);

        \def\w{0.6} 

        \filldraw[color=ao(english), fill=ao(english)!70, thick] 
    ($(a)+(60:1cm)$)
    .. controls ($ (ab)!\w cm!90:(a)$).. ($(b)+(0:1cm)$)
    arc (360:120:1)
    .. controls ($ (bc)!\w cm!90:(b)$)..  ($(c)+(60:1cm)$)
    arc (60:-180:1)
     .. controls ($ (cd)!\w cm!90:(c)$)..  ($(d)+(120:1cm)$)
     arc (120:-120:1)
     .. controls ($ (de)!\w cm!90:(d)$)..  ($(e)+(180:1cm)$)
     arc (180:-60:1)
     .. controls ($ (ef)!\w cm!90:(e)$)..  ($(f)+(240:1cm)$)
     arc (240:0:1)
     .. controls ($ (fa)!\w cm!90:(f)$)..  ($(a)+(300:1cm)$)
     arc (300:60:1)
    ;

        \draw[thick]  (a) circle (1cm);
        \draw[thick]  (b) circle (1cm);
        \draw[thick]  (c) circle (1cm);
        \draw[thick]  (d) circle (1cm);
        \draw[thick]  (e) circle (1cm);
        \draw[thick]  (f) circle (1cm);

\end{tikzpicture}
}}

\newcommand{\BEsevenfour}{\vc{
\begin{tikzpicture}
\def\r{3} 
\def\a{51.428} 
\def\b{25.714} 

        \coordinate (O) at (0,0);
        \coordinate (a) at ($(O)+(0:\r cm)$);
        \coordinate (b) at ($(O)+(\a:\r cm)$);
        \coordinate (c) at ($(O)+(2*\a:\r cm)$);
        \coordinate (d) at ($(O)+(3*\a:\r cm)$);
        \coordinate (e) at ($(O)+(4*\a:\r cm)$);
        \coordinate (f) at ($(O)+(5*\a:\r cm)$);
        \coordinate (g) at ($(O)+(6*\a:\r cm)$);

        \coordinate (ab) at ($ (a)!.5!(b)$);
        \coordinate (bc) at ($ (b)!.5!(c)$);
        \coordinate (cd) at ($ (c)!.5!(d)$);
        \coordinate (de) at ($ (d)!.5!(e)$);
        \coordinate (ef) at ($ (e)!.5!(f)$);
        \coordinate (fg) at ($ (f)!.5!(g)$);
        \coordinate (ga) at ($ (g)!.5!(a)$);

        \def\w{0.6} 

        \filldraw[color=ao(english), fill=ao(english)!70, thick] 
    ($(a)+(30+\b:1cm)$)
    .. controls ($ (ab)!\w cm!90:(a)$).. ($(b)+(\b-30:1cm)$)
    arc (330+\b:30+3*\b:1)
    .. controls ($ (bc)!\w cm!90:(b)$)..  ($(c)+(3*\b-30:1cm)$)
    arc (3*\b+330:30+5*\b:1)
     .. controls ($ (cd)!\w cm!90:(c)$)..  ($(d)+(150-\a:1cm)$)
     arc (150-\a:-150:1)
     .. controls ($ (de)!\w cm!90:(d)$)..  ($(e)+(150:1cm)$)
     arc (150:\a-150:1)
     .. controls ($ (ef)!\w cm!90:(e)$)..  ($(f)+(330-5*\b:1cm)$)
     arc (330-5*\b:30-3*\b:1)
     .. controls ($ (fg)!\w cm!90:(f)$)..  ($(g)+(330-3*\b:1cm)$)
     arc (330-3*\b:30-\b:1)
     .. controls ($ (ga)!\w cm!90:(g)$).. ($(a)+(330-\b:1cm)$)
    arc (330-\b:30+\b:1)
    ;

    \filldraw[color = szurke, fill=szurke!90, very thick]
    ($(e)+(30:1cm)$)
    .. controls ($ (de)!\w cm!90:(e)$)..  ($(d)+(-30:1cm)$)
    arc (330:210:1)
    .. controls ($ (de)!\w cm!90:(d)$)..  ($(e)+(150:1cm)$)
    arc (150:60:1)
    ;

    \filldraw[color = szurke, fill=szurke!90, very thick]
    ($(b)+(30+3*\b:1cm)$)
    .. controls ($ (bc)!\w cm!90:(b)$)..  ($(c)+(3*\b-30:1cm)$)
    arc (3*\b-30:3*\b-150:1)
    .. controls ($ (bc)!\w cm!90:(c)$)..  ($(b)+(150+3*\b:1cm)$)
    arc (150+3*\b:30+3*\b:1)
    ;

    \filldraw[color = szurke, fill=szurke!90, very thick]
    ($(g)+(330-3*\b:1cm)$)
    .. controls ($ (fg)!\w cm!90:(f)$)..  ($(f)+(30-3*\b:1cm)$)
    arc (30-3*\b:150-3*\b:1)
    .. controls ($ (fg)!\w cm!90:(g)$)..  ($(g)+(210-3*\b:1cm)$)
    arc (210-3*\b:330-3*\b:1)
    ;

        \draw[thick]  (a) circle (1cm);
        \draw[thick]  (b) circle (1cm);
        \draw[thick]  (c) circle (1cm);
        \draw[thick]  (d) circle (1cm);
        \draw[thick]  (e) circle (1cm);
        \draw[thick]  (f) circle (1cm);
        \draw[thick]  (g) circle (1cm);

\end{tikzpicture}
}}

\newcommand{\BEfiveprofile}{\vc{
\begin{tikzpicture}
\def\r{2.2} 
        \coordinate (O) at (0,0);
        \coordinate (a) at ($(O)+(0:\r cm)$);
        \coordinate (b) at ($(O)+(72:\r cm)$);
        \coordinate (c) at ($(O)+(144:\r cm)$);
        \coordinate (d) at ($(O)+(216:\r cm)$);
        \coordinate (e) at ($(O)+(288:\r cm)$);

        \coordinate (ab) at ($ (a)!.5!(b)$);
        \coordinate (bc) at ($ (b)!.5!(c)$);
        \coordinate (cd) at ($ (c)!.5!(d)$);
        \coordinate (de) at ($ (d)!.5!(e)$);
        \coordinate (ea) at ($ (e)!.5!(a)$);

        \coordinate (be) at ($ (b)!.5!(e)$);
    
        \def\w{0.6} 

        \filldraw[color=ao(english), fill=ao(english)!70, thick] 
    ($(a)+(66:1cm)$)
    .. controls ($ (ab)!\w cm!90:(a)$).. ($(b)+(6:1cm)$)
    arc (366:138:1)
    .. controls ($ (bc)!\w cm!90:(b)$)..  ($(c)+(78:1cm)$)
    arc (78:-150:1)
     .. controls ($ (cd)!\w cm!90:(c)$)..  ($(d)+(150:1cm)$)
     arc (150:-78:1)
     .. controls ($ (de)!\w cm!90:(d)$)..  ($(e)+(222:1cm)$)
     arc (222:-6:1)
     .. controls ($ (ea)!\w cm!90:(e)$)..  ($(a)+(294:1cm)$)
     arc (294:66:1)
    ;

    \filldraw[color = szurke, fill=szurke!90, very thick]
    ($(a)+(66:1cm)$)
    .. controls ($ (ab)!\w cm!90:(a)$).. ($(b)+(6:1cm)$)
    arc (366:210:1)
    .. controls ($ (be)!0.35*\w cm!90:(b)$)..  ($(e)+(150:1cm)$)
    arc (150:-6:1)
    .. controls ($ (ea)!\w cm!90:(e)$)..  ($(a)+(294:1cm)$)
     arc (294:66:1)
    ;

        \draw[thick]  (a) circle (1cm);
        \draw[thick]  (b) circle (1cm);
        \draw[thick]  (c) circle (1cm);
        \draw[thick]  (d) circle (1cm);
        \draw[thick]  (e) circle (1cm);

\end{tikzpicture}
}}

\newcommand{\conster}{\vc{
\begin{tikzpicture}
\def\r{3} 
        \coordinate (O) at (0,0);
        \coordinate (a) at ($(O)+(45:\r cm)$);
        \coordinate (b) at ($(O)+(135:\r cm)$);
        \coordinate (c) at ($(O)+(225:\r cm)$);
        \coordinate (d) at ($(O)+(315:\r cm)$);

        \coordinate (ab) at ($ (a)!.5!(b)$);
        \coordinate (bc) at ($ (b)!.5!(c)$);
        \coordinate (cd) at ($ (c)!.5!(d)$);
        \coordinate (da) at ($ (d)!.5!(a)$);
    
        \def\w{0.6} 

    \filldraw[color = szurke, fill=szurke!70, very thick]
    ($(a)+(-15:1cm)$)
    .. controls ($ (da)!\w cm!90:(d)$)..  ($(d)+(30:1.3cm)$)
     arc (30:150:1.3)
     .. controls ($ (da)!\w cm!90:(a)$)..  ($(a)+(195:1cm)$)
     arc (195:345:1)
    ;

    \filldraw[color = szurke, fill=szurke!70, very thick]
    ($(b)+(-15:1cm)$)
    .. controls ($ (bc)!\w cm!90:(c)$)..  ($(c)+(30:1.3cm)$)
     arc (30:150:1.3)
     .. controls ($ (bc)!\w cm!90:(b)$)..  ($(b)+(195:1cm)$)
     arc (195:345:1)
    ;

    \filldraw[color = ao(english), fill=ao(english)!60, very thick]
    ($(b)+(0:1cm)$)
    .. controls ($ (O)!\w cm!90:(d)$)..  ($(d)+(90:1.3cm)$)
     arc (90:180:1.3)
     .. controls ($ (O)!\w cm!90:(b)$)..  ($(b)+(270:1cm)$)
     arc (270:360:1)
    ;

    \filldraw[color = ao(english), fill=ao(english)!60, very thick]
    ($(b)+(0:1cm)$)
    .. controls ($ (O)!\w cm!90:(d)$)..  ($(d)+(90:1.3cm)$)
     arc (90:180:1.3)
     .. controls ($ (O)!\w cm!90:(b)$)..  ($(b)+(270:1cm)$)
     arc (270:360:1)
    ;

    \filldraw[color = ao(english), fill=ao(english)!60, very thick]
    ($(a)+(180:1cm)$)
    .. controls ($ (O)!\w cm!90:(a)$)..  ($(c)+(90:1.3cm)$)
     arc (90:0:1.3)
     .. controls ($ (O)!\w cm!90:(c)$)..  ($(a)+(270:1cm)$)
     arc (270:180:1)
    ;

    \filldraw[color = ao(english), fill=ao(english)!60, very thick]
    ($(c)+(45:1.3cm)$)
    .. controls ($ (cd)!\w cm!90:(d)$)..  ($(d)+(135:1.3cm)$)
     arc (135:225:1.3)
     .. controls ($ (cd)!\w cm!90:(c)$)..  ($(c)+(315:1.3cm)$)
     arc (-45:45:1.3)
    ;

        \draw[very thick]  (a) circle (1cm);
        \draw[very thick]  (b) circle (1cm);
        \draw[very thick]  (c) circle (1.3cm);
        \draw[very thick]  (d) circle (1.3cm);

    \node at (a) {\Large $V_6$};
    \node at (b) {\Large $V_3$};
    \node at (c) {\Large $V_1\cup V_2$};
    \node at (d) {\Large $V_4\cup V_5$};

\end{tikzpicture}
}}

\newcommand{\conjec}{\vc{
\begin{tikzpicture}
\def\r{3} 
        \coordinate (O) at (0,0);
        \coordinate (a) at ($(O)+(45:\r cm)$);
        \coordinate (b) at ($(O)+(135:\r cm)$);
        \coordinate (c) at ($(O)+(225:\r cm)$);
        \coordinate (d) at ($(O)+(315:\r cm)$);

        \coordinate (ab) at ($ (a)!.5!(b)$);
        \coordinate (bc) at ($ (b)!.5!(c)$);
        \coordinate (cd) at ($ (c)!.5!(d)$);
        \coordinate (da) at ($ (d)!.5!(a)$);
    
        \def\w{0.6} 

    \filldraw[color = szurke, fill=szurke!70, very thick]
    ($(b)+(-15:1.2cm)$)
    .. controls ($ (bc)!\w cm!90:(c)$)..  ($(c)+(30:1.2cm)$)
     arc (30:150:1.2)
     .. controls ($ (bc)!\w cm!90:(b)$)..  ($(b)+(195:1.2cm)$)
     arc (195:345:1.2)
    ;

    \filldraw[color = ao(english), fill=ao(english)!60, very thick]
    ($(b)+(0:1.2cm)$)
    .. controls ($ (O)!\w cm!90:(d)$)..  ($(d)+(90:1.2cm)$)
     arc (90:180:1.2)
     .. controls ($ (O)!\w cm!90:(b)$)..  ($(b)+(270:1.2cm)$)
     arc (270:360:1.2)
    ;

    \filldraw[color = ao(english), fill=ao(english)!60, very thick]
    ($(c)+(45:1.2cm)$)
    .. controls ($ (cd)!\w cm!90:(d)$)..  ($(d)+(135:1.2cm)$)
     arc (135:225:1.2)
     .. controls ($ (cd)!\w cm!90:(c)$)..  ($(c)+(315:1.2cm)$)
     arc (-45:45:1.2)
    ;

        \draw[very thick]  (b) circle (1.2cm);
        \draw[very thick]  (c) circle (1.2cm);
        \draw[very thick]  (d) circle (1.2cm);

\end{tikzpicture}
}}

\definecolor{ao(english)}{rgb}{0.0,0.5,0.0}

\newcommand{\vc}[1]{\ensuremath{\vcenter{\hbox{#1}}}}

\newcommand\rt{\ensuremath{\mathrm{\bf RT}}}
\newcommand\be{\ensuremath{\mathrm{\bf \mathcal{G}}}}

\author{J\'ozsef Balogh\thanks{Department of Mathematics, University of Illinois at Urbana-Champaign, IL, 
USA. Email: \texttt{jobal@illinois.edu}.
Balogh was supported in part by NSF grants DMS-1764123 and RTG DMS-1937241, FRG DMS-2152488, the Arnold O. Beckman Research Award (UIUC Campus Research Board RB 24012).}
\qquad
Van Magnan\thanks{Department of Mathematical Sciences, University of Montana. Email: \texttt{van.magnan@umontana.edu}.}
\qquad
Cory Palmer\thanks{Department of Mathematical Sciences, University of Montana. Email: \texttt{cory.palmer@umontana.edu}.
Research supported by a grant from the Simons Foundation \#712036.

}}

\title{Generalized Ramsey-Tur\'an Numbers}

\begin{document}

\maketitle

\begin{abstract}
The Ramsey-Tur\'an problem for $K_p$ asks for the maximum number of edges in an $n$-vertex $K_p$-free graph with independence number $o(n)$. In a natural generalization of the problem,  cliques larger than the edge $K_2$ are counted. Let {\bf RT}$(n,\#K_q,K_p,o(n))$
denote the maximum number of copies of $K_q$ in an $n$-vertex $K_p$-free graph with independence number $o(n)$. Balogh, Liu and Sharifzadeh determined the asymptotics of {\bf RT}$(n,\# K_3,K_p,o(n))$. In this paper we will establish the asymptotics for counting copies of $K_4$, $K_5$, and for the case  $p \geq 5q$. We also provide a family of counterexamples to a conjecture of Balogh, Liu and Sharifzadeh.
\end{abstract}

\section{Introduction}

The foundational result in extremal graph theory is Tur\'an's theorem, which determines the maximum number of edges in an $n$-vertex graph with no clique $K_p$ as a subgraph. The unique \emph{extremal graph} attaining this maximum is the \emph{Tur\'an graph} $T(n,p-1)$, the complete balanced $(p-1)$-partite graph on $n$ vertices. The \emph{Tur\'an number} $\ex(n,F)$ of a graph $F$ is the maximum number of edges in an \emph{$F$-free} (i.e., no $F$ subgraph) $n$-vertex graph. The study of the function $\ex(n,F)$ is a central pursuit in extremal graph theory. When $F$ has chromatic number  $p \geq 3$, then the fundamental Erd\H os-Stone-Simonovits theorem \cite{ErSt, ErSi} gives $\ex(n,F) = (1+o(1))\ex(n,K_p)$. When $F$ has chromatic number $2$, then many questions remain open as detailed in the survey of F\"uredi and Simonovits~\cite{FuSi}.

There are many natural generalizations of the Tur\'an number; this paper is concerned with a common generalization of two such well-studied notions. Instead of counting the number of edges in an $F$-free graph, one may count the number of copies of some subgraph $H$. For example, the \emph{Erd\H os pentagon problem} asks for the maximum number of copies of the cycle $C_5$ in a triangle-free graph (see \cite{LiPf} for a history). Of particular importance here is Zykov's theorem \cite{zykov} which determines the maximum number of $K_q$ copies  in an $n$-vertex $K_p$-free graph. More generally, Alon and Shikhelman~\cite{AlSh} initiated the study of the general function
$\ex(n,\#H,F)$, i.e., the maximum number of $H$ copies in an $n$-vertex $F$-free graph.

Because the Tur\'an graph has independent sets of linear size, it is reasonable to ask for the maximum number of edges in a $K_p$-free graph with sublinear size independent sets. The \emph{Ramsey-Tur\'an problem} (likely first posed by Andr\'asfai~\cite{andras}) asks to determine $\rt(n,K_p,o(n))$, the maximum  number of edges in a $K_p$-free $n$-vertex graph with independence number $o(n)$.  Erd\H os and S\'os~\cite{ErSoRT} determined $\rt(n,K_p,o(n))$ for odd cliques $K_p$. When forbidding $K_4$, Szemer\'edi~\cite{Sz4} proved an upper bound of $\frac{1}{8}n^2+o(n^2)$ which was later matched by a construction of Bollob\'as and Erd\H os~\cite{BE-const}. The remaining even clique cases were determined by Erd\H os, Hajnal, S\'os and Szemer\'edi~\cite{EHSSz} (see Theorem~\ref{edge-thm} below for a summary of these results).
The survey of Simonovits and S\'os~\cite{rt-survey} includes a detailed history of Ramsey-Tur\'an problems. See \cite{BCMM, BaloghLenz, balogh2023weighted, liu2021geometric} for a variety of recent developments on problems in the area.

Combining the generalized Tur\'an and Ramsey-Tur\'an problem leads naturally to the following setting introduced by Balogh, Liu, Sharifzadeh \cite{balogh2017problems}. 
For graphs $H$ and $F$, the {\it generalized Ramsey-Tur\'an number},
\[
\rt(n, \#H, F, f(n)),
\]
is the maximum number of copies of $H$ in an $F$-free $n$-vertex graph with no independent set of size greater than $f(n)$.
The traditional Ramsey-Tur\'an function is when $H$ is an edge, $F$ is a complete graph and $f(n)=o(n)$.

Our goal is to find the asymptotics of $\rt(n, \#K_q, K_p, \alpha n)$. 
We use the mnemonic to ``quantify'' the $q$-clique $K_q$ while ``prohibiting'' the $p$-clique $K_p$.
\begin{defn}
The \emph{$K_q$-Ramsey-Tur\'an density} is the limit
    \[
    \mathfrak{R}_q(p) := \lim_{\alpha \to 0} \lim_{n \to \infty} \frac{\rt(n, \#K_q, K_p, \alpha n)}{n^q}.
    \]    
\end{defn}

In this terminology the original Ramsey-Tur\'an results for cliques are captured in the following theorem.

\begin{thm}[Erd\H os, Hajnal, S\'os, Szemer\'edi \cite{EHSSz}]\label{edge-thm}
If $p=2h\geq 4$ is even, then
    \[
    \mathfrak{R}_2(2h) =
    \frac{3h-5}{6h-4}.
    \]
If $p=2h+1\geq 5$ is odd, then
    \[
    \mathfrak{R}_2(2h+1) =  
       \left(\frac{1}{h}\right)^2 \binom{h}{2}.
    \]    
\end{thm}

A main result of \cite{balogh2017problems} extends this classic theorem  to counting triangles.

\begin{thm}[Balogh, Liu, Sharifzadeh \cite{balogh2017problems}]\label{triangle-thm}
If $p=2h \geq 6$ is even, then
    \[
    \mathfrak{R}_3(2h) = \max_{0 \leq x \leq 1}
       \frac{1}{2}\left(\frac{x}{2}\right)^2(1-x)+x\binom{h-2}{2}\left(\frac{1-x}{h-2}\right)^2 + \binom{h-2}{3}\left(\frac{1-x}{h-2}\right)^3.
    \]  
If $p = 2h+1\geq 7$ is odd, then
    \[
    \mathfrak{R}_3(2h+1) =  
       \left(\frac{1}{h}\right)^3 \binom{h}{3}.
    \]    
\end{thm}

Our first two main theorems extend this to counting cliques $K_4$ and $K_5$. Note that the bounds continue the same discrepancy between the $p$ even and odd cases. However, for $q\geq 4$ and $p$ small $\mathfrak{R}_q(p)$ exhibits distinct behavior. For counting copies of $K_4$ we prove:

\begin{thm}\label{k4-thm}
If $6 \leq p \leq 8$, then
\[
\mathfrak{R}_4(6) = \left(\frac{1}{4}\right)^4\left(\frac{1}{2}\right)^6,  \qquad    \mathfrak{R}_4(7) =  \left(\frac{1}{4}\right)^4\left(\frac{1}{2}\right)^2,  \qquad \mathfrak{R}_4(8) = \left(\frac{1}{4}\right)^4\left(\frac{1}{2}\right).
\]
If $p=2h \geq 8$ is even, then
    \[
    \mathfrak{R}_4(2h) = \max_{0 \leq x \leq 1}
       \frac{1}{2}\left(\frac{x}{2}\right)^2\binom{h-2}{2}\left(\frac{1-x}{h-2}\right)^2+x\binom{h-2}{3}\left(\frac{1-x}{h-2}\right)^3 + \binom{h-2}{4}\left(\frac{1-x}{h-2}\right)^4.
    \]
    If $p = 2h+1\geq 9$ is odd, then
    \[
    \mathfrak{R}_4(2h+1) = \left(\frac{1}{h}\right)^4\binom{h}{4}.       
    \]    
\end{thm}

 For the number of copies of $K_5$, we prove:

\begin{thm}\label{k5-thm}
If $7 \leq p \leq 11$, then
\[
\mathfrak{R}_5(7) =\left(\frac{1}{5}\right)^5\left(\frac{1}{2}\right)^{10},   \qquad \mathfrak{R}_5(8) = \left(\frac{1}{5}\right)^5\left(\frac{1}{2}\right)^{4},     
\]
\[
\qquad \mathfrak{R}_5(9) = \left(\frac{1}{5}\right)^5\left(\frac{1}{2}\right)^{2}, \qquad  \mathfrak{R}_5(10) = \binom{6}{5}\left(\frac{1}{6}\right)^5\left(\frac{1}{2}\right)^{2},  
\]
\[
\mathfrak{R}_5(11) = \max_{0 \leq x \leq 1} \left(\frac{x}{4}\right)^4\left(\frac{1}{2}\right)^2(1-x) + 4\left(\frac{x}{4}\right)^3\left(\frac{1}{2}\right)\left(\frac{1-x}{2}\right)^2 = \frac{675+228\sqrt{15}}{480200}.
\]

If $p=2h \geq 12$ is even, then
    \[
    \mathfrak{R}_5(2h) = \max_{0 \leq x \leq 1}
       \frac{1}{2}\left(\frac{x}{2}\right)^2\binom{h-2}{3}\left(\frac{1-x}{h-2}\right)^3+x\binom{h-2}{4}\left(\frac{1-x}{h-2}\right)^4 + \binom{h-2}{5}\left(\frac{1-x}{h-2}\right)^5.
    \]
If $p = 2h+1\geq 13$ is odd, then
    \[
    \mathfrak{R}_5(2h+1) = \left(\frac{1}{h}\right)^5\binom{h}{5}.       
    \]    
\end{thm}

 When $p=q+1$, then it is easy to see that $\mathfrak{R}_q(q+1) = 0$. Indeed, as there is no $(q+1)$-clique, the common neighborhood of any $q-1$ vertices is an independent set and thus has sublinear order. Therefore, the number of $K_q$ copies is $o(n^q)$ in such an $n$-vertex graph.
In~\cite{balogh2017problems}, the authors determined $\mathfrak{R}_q(q+2)$. We extend this theorem.

\begin{thm}\label{q-plus}
\begin{enumerate}
    \item[(a)]

For $q\geq 2$,
    \[
    \mathfrak{R}_{q}(q+2) = \left(\frac{1}{q}\right)^q\left(\frac{1}{2}\right)^{\binom{q}{2}}
 \qquad \text{and} \quad
    \mathfrak{R}_{q}(q+3) =  \left(\frac{1}{q}\right)^q\left(\frac{1}{2}\right)^{\binom{\lfloor q/2\rfloor}{2}+ \binom{\lceil q/2 \rceil}{2}}.
    \]
\item[(b)] For $q \geq 3$, 
    \[
    \mathfrak{R}_{q}(q+4) = \left(\frac{1}{q}\right)^q\left(\frac{1}{2}\right)^{\binom{\lfloor q/3\rfloor}{2}+ \binom{\lfloor (q+1)/3 \rfloor}{2} +\binom{\lfloor (q+2)/3 \rfloor}{2}}.
    \]
\end{enumerate}
\end{thm}

Observe Theorems~\ref{edge-thm}, \ref{triangle-thm}, \ref{k4-thm}, and \ref{k5-thm} suggest a ``periodic behavior'' based on the parity of $p$ for $\mathfrak{R}_q(p)$ when $p$ is large enough compared to $q$.
In~\cite{balogh2017problems}, the authors conjectured that this behavior should occur for all $p \geq 2q+1$. In Section~\ref{section-counter} we show that this conjecture does not hold in general, but we can prove that it holds for $p \geq 5q$.

\begin{thm}\label{general-q}
        Let $q\geq 5$ and $p\geq 5q$.
        If $p=2h \geq 5q$ is even, then
    \[
    \mathfrak{R}_q(2h) = \max_{0 \leq x \leq 1}
       \frac{1}{2}\left(\frac{x}{2}\right)^2\binom{h-2}{3}\left(\frac{1-x}{h-2}\right)^{q-2}+x\binom{h-2}{4}\left(\frac{1-x}{h-2}\right)^{q-1} + \binom{h-2}{5}\left(\frac{1-x}{h-2}\right)^q.
    \]
If $p = 2h+1\geq 5q$ is odd, then
    \[
    \mathfrak{R}_q(2h+1) = \left(\frac{1}{h}\right)^q\binom{h}{5}.       
    \]    
\end{thm}

{\bf Organization. }
In Section~\ref{construct} we describe a general construction that will be used to prove the lower bounds in the stated theorems. In Section~\ref{section-counter} we discuss counterexamples to a conjecture stated in \cite{balogh2017problems} on the behavior of $\mathfrak{R}_q(p)$. In Section~\ref{section-zykov} we introduce our main tool: a weighted version of Zykov's theorem. In Section~\ref{section-proofs} we establish the upper bounds in Theorems~\ref{edge-thm},~\ref{triangle-thm},~\ref{k4-thm}, \ref{k5-thm}, \ref{q-plus}, and \ref{general-q}.

\smallskip

{\bf Notation and alphabet. } Our notation is standard. For undefined terms we refer the reader to the monograph~\cite{bo-book}. 
We have made an effort to keep notation consistent across the manuscript sections. For example, $n$ will be the number of vertices of a graph and the parameters $p$ and $q$ denote the order of a ``prohibited'' clique and ``quantified'' clique, respectively. In a multipartite graph, the collection of classes  of pairwise density $1$ is $\mathcal{T}$ with $t=|\mathcal{T}|$; we call these classes \emph{par{\underline t}s}. We use $\mathcal{S}$ with $s=|\mathcal{S}|$ for the collection of classes of pairwise density at least $1/2$; we call these classes \emph{{\underline c}ells}. Generic integers are $k$ and $\ell$ while $u$ and $v$ are generic vertices. Typically, $x$ (and $y$) are the parameter(s) of an optimization---usually the number of vertices in some part or cell. 
Finally, we use the standard  $\varepsilon$ and $M=M(\varepsilon)$ when applying the Regularity Lemma and the associated reduced graph $R$ will have $r$ vertices.

\smallskip

{\bf Note.} While this manuscript was being finalized we learned that Gao, Jiang, Liu and Sankar~\cite{gao2024generalized} independently proved some related results.

\section{Constructions}\label{construct}

First, let us introduce a classical construction that is central in this work:

{\bf Bollob\'as-Erd\H os graph~\cite{BoEr}.} 
Fix $\varepsilon>0$ and let $d$ and even $n$ be sufficiently large integers. Put
 $\mu=\varepsilon/\sqrt d$. Partition the $d $-dimensional 
unit sphere $\mathbb{S}^d$ into $n/2$ domains, $D_1,\dots,D_{n/2}$, of equal measure with diameter (i.e., the maximum distance between any two points) less than $\mu/2$. For every $1\le i\le n/2$, choose two points $u_i, v_i\in D_i$.  Let $U=\{u_1,\dots,u_{n/2}\}$ and $V=\{v_1,\dots,v_{n/2}\}$. Let $\textbf{BE}(U,V)$ be the graph with vertex set $U\cup V$ and edge set as follows. For every $u,u'\in U$ and $v,v'\in V$,
\begin{enumerate}
\item  $uv\in E(\textbf{BE}(U,V))$ if their distance is less than $\sqrt 2-\mu$,
\item  $uu'\in E(\textbf{BE}(U,V))$ if their distance is more than $2-\mu$,
\item  $vv'\in E(\textbf{BE}(U,V))$ if their distance is more than $2-\mu$.
\end{enumerate} 


Crucially, Bollob\'as and Erd\H os~\cite{BoEr} showed that an $n$-vertex $\textbf{BE}(U,V)$ is $K_4$-free, has independence number $o(n)$ and has $\frac{1}{8}n^2+o(n^2)$ edges, providing a lower bound on  $\rt(n,K_4,o(n))$. Also note that the number of edges in both $U$ and $V$ is $o(n^2)$, and each vertex has $\left(\frac{1}{2}-o(1)\right) \frac{n}{2}$ neighbors in the opposite class.

\bigskip

We now extend the the Bollob\'as-Erd\H os graph above to be multipartite. This generalization can be found in \cite{balogh2017problems}; we copy their description (and lower bound on clique counts in such graphs) here, for completeness.

{\bf $s$-Bollob\'as-Erd\H os graph.} Let $D_1,\dots, D_{n/s}$ be a partition of the high-dimensional unit sphere of equal measure as in the Bollob\'as-Erd\H os graph construction. Let $H$ be an $n$-vertex graph with a balanced partition $V_1,\ldots, V_s$, where each $V_i$ consists of one point from each of the $n/s$ domains $D_1,\ldots, D_{n/s}$. For every pair of distinct integers $V_i,V_j$, let $H[V_i\cup V_j]$ be a copy of $\textbf{BE}(V_i,V_j)$. Note that each $H[V_i]$ is triangle-free and each $H[V_i\cup V_j]$ is $K_4$-free. We claim $H$ is $K_{s+2}$-free. Indeed, any copy of $K_{s+2}$ would have four vertices forming a $K_4$ spanned by two parts $V_i, V_j$ of $H$. But this would give a $K_4$  in the Bollob\'as-Erd\H os graph $\textbf{BE}(V_i,V_j)$, a contradiction.

 We count the number of $K_\ell$ copies with one vertex in each $V_1,V_2,\dots, V_\ell$. 
 Fix a vertex $v_1\in V_1$, a uniformly at random chosen $v_2\in V_2$ is adjacent to $v_1$ if $v_2$ is in the cap (almost a hemisphere) centered at $v_1$ with measure $\frac12-o(1)$, which happens with probability $\frac12-o(1)$. Now we fix a clique on vertex set $v_1, v_2,\ldots, v_{i-1}$ with $i\geq 2$ and $v_j\in V_j$. 
 One can prove that the number of vertices in $V_i$ that are in $\bigcap_{j=1}^{i-1}N(v_j)$ is at least $\left(\left(\frac{1}{2}\right)^{i-1}-o(1)\right)\frac{n}{s}$. 
 There are $\binom{s}{\ell}$ ways to select $\ell$ classes from $s$, so the number  of $K_\ell$ in $H$ satisfies
\begin{equation*}
\N_\ell(H)\geq \binom{s}{\ell} \prod_{i=1}^\ell\left(\left(\frac{1}{2}\right)^{i-1}-o(1)\right)\frac{n}{s}=\left(1+o(1)\right)\binom{s}{\ell}\left(\frac{1}{2}\right)^{\binom{\ell}{2}}\left(\frac{n}{s}\right)^\ell.    
\end{equation*}

Note that a $2$-Bollob\'as-Erd\H os graph is simply an ordinary Bollob\'as-Erd\H os graph. For convenience we take the degenerate $1$-Bollob\'as-Erd\H os graph to be one class of a Bollob\'as-Erd\H os graph, i.e., a set of vertices spanning a  triangle-free graph  with sublinear independence number.

We now describe the main construction that we will use to attain lower bounds on $\mathfrak{R}_q(p)$ in Theorems~\ref{edge-thm},~\ref{triangle-thm},~\ref{k4-thm}, \ref{k5-thm}, ~\ref{q-plus}, and \ref{general-q}. 
It is a careful gluing of copies of $\ell$-Bollob\'as-Erd\H os graphs for appropriate values of $\ell$.

\begin{const*} Fix $s \geq t$.
Define a family of graphs $\mathcal{G}(n;s,t)$ as follows. Put $r = s \mod t$.
    Begin with an $n$-vertex complete $t$-partite graph and into each of the $r$ parts embed a $\lceil s/t \rceil$-Bollob\'as-Erd\H os graph and to each of the remaining $t-r$ parts embed a $\lfloor s/t \rfloor$-Bollob\'as-Erd\H os graph. 
\end{const*}

\begin{figure}[H]
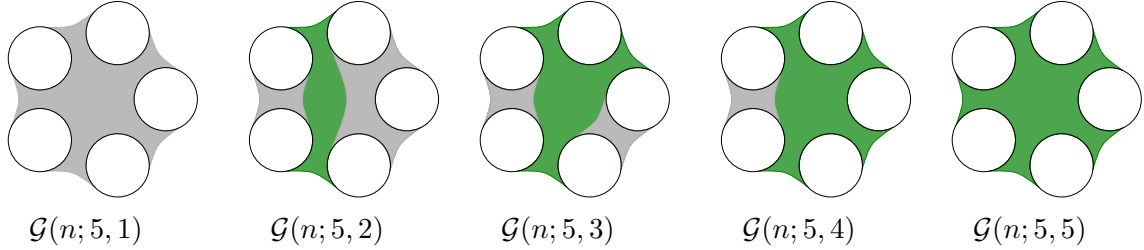

\centering
\begin{subfigure}{0.14\textwidth}
\resizebox{2.7cm}{!}{
\BEfiveone
}
\caption*{$\mathcal{G}(n;5,1)$}
\end{subfigure}
\hspace{1.5em}
\begin{subfigure}{0.14\textwidth}
\resizebox{2.7cm}{!}{
\BEfivetwo
}
\caption*{$\mathcal{G}(n;5,2)$}
\end{subfigure}%
\hspace{1.5em}
\begin{subfigure}{0.14\textwidth}
\resizebox{2.7cm}{!}{
\BEfivethree
}
\caption*{$\mathcal{G}(n;5,3)$}
\end{subfigure}
\hspace{1.5em}
\begin{subfigure}{0.14\textwidth}
\resizebox{2.7cm}{!}{
\BEfivefour
}
\caption*{$\mathcal{G}(n;5,4)$}
\end{subfigure}%
\hspace{1.5em}
\begin{subfigure}{0.14\textwidth}
\resizebox{2.7cm}{!}{
\BEfivefive
}
\caption*{$\mathcal{G}(n;5,5)$}
\end{subfigure}
\caption{Green (dark grey) indicates edge density $1$ and grey indicates edge density $1/2-o(1)$.}
\end{figure}

A graph in $\mathcal{G}(n;s,t)$ has two natural vertex partitions. First, the initial $t$-partite graph has $t$ classes, that are called \emph{parts}. The edge density between pairs of parts is $1$. Within each part, there is a partition into $\lceil s/t \rceil$ or $\lfloor s/t \rfloor$ classes that we call \emph{cells}. Within a part, the edge density between pairs of cells is $1/2-o(1)$. Note that the partition into cells is a refinement of the partition into parts.

The initial $t$-partite graph in $\mathcal{G}(n;s,t)$ is not necessarily balanced. For fixed $s$, we will assume that the $r$ parts with $\lceil s/t\rceil$ cells have cardinality $x$ and the $t-r$ parts with $\lfloor s/t \rfloor$ cells have cardinality $\frac{n-rx}{t-r}$. Moreover, two cells in the same part have the same cardinality. From here it is not difficult (but impractical) to describe an optimization in $x$ that will determine (asymptotically) the maximum possible number  of $K_q$. Generally, we will assume $x$ is optimized to maximize the number  of $K_q$.

Each $\lceil s/t\rceil$- and $\lfloor s/t \rfloor$-Bollob\'as-Erd\H os graph has sublinear independence number. As there are constant $t$ many parts, every graph in $\mathcal{G}(n;s,t)$ has independence number $o(n)$.
Moreover, the size of a largest clique in a graph in $\mathcal{G}(n;s,t)$ is at most $(t-r)\left( \lfloor s/t \rfloor + 1 \right) + r \left( \lceil s/t \rceil +1\right) = s+t$.

Some notable instances of $\mathcal{G}(n;s,t)$ are $\mathcal{G}(n;s,s)$ which is simply a complete $s$-partite graph, $\mathcal{G}(n;s,s-1)$ which is a complete $s$-partite graph with one pair of parts replaced by a Bollob\'as-Erd\H os graph,
and $\mathcal{G}(n;s,1)$ which is an $s$-Bollob\'as-Erd\H os graph. 

The instances of $\mathcal{G}(n;s,t)$ that appear in Table~\ref{big-table} can be checked that they match the corresponding values of $\mathfrak{R}_q(p)$ in Theorems~\ref{edge-thm},~\ref{triangle-thm},~\ref{k4-thm}, and \ref{k5-thm} thus providing lower bounds in those proofs.

 \begin{table}
\centering
\begin{tabular}{|Sr|Sc|Sc|Sc| Sc |}\hline
\diagbox[width=6em]{$p$}{$q$}
 &  2 & 3 & 4 & 5 \\ \hline
4 & \adjustbox{max height=0.34in, max width = 0.8in}\BEtwoone & & & \\ \hline
 5 & \adjustbox{max height=0.34in, max width = 0.8in}\BEtwotwo & \adjustbox{max height=0.34in, max width = 0.8in}\BEthreeone & & \\
  \hline
6 &\adjustbox{max height=0.34in, max width = 0.8in}\BEthreetwo & \adjustbox{max height=0.34in, max width = 0.8in}\BEthreetwo & \adjustbox{max height=0.34in, max width = 0.8in}\BEfourone & \\
  \hline
7&\adjustbox{max height=0.34in, max width = 0.8in}\BEthreethree  & \adjustbox{max height=0.34in, max width = 0.8in}\BEthreethree & \adjustbox{max height=0.34in, max width = 0.8in}\BEfourtwo & \adjustbox{max height=0.34in, max width = 0.8in}\BEfiveone \\ \hline
8 &\adjustbox{max height=0.34in, max width = 0.8in}\BEfourthree  & \adjustbox{max height=0.34in, max width = 0.8in}\BEfourthree & \adjustbox{max height=0.34in, max width = 0.8in}\BEfourthree & \adjustbox{max height=0.34in, max width = 0.8in}\BEfivetwo \\ 
  \hline
9 & \adjustbox{max height=0.34in, max width = 0.8in}\BEfourfour & \adjustbox{max height=0.34in, max width = 0.8in}\BEfourfour & \adjustbox{max height=0.34in, max width = 0.8in}\BEfourfour & \adjustbox{max height=0.34in, max width = 0.8in}\BEfivethree \\ 
  \hline
10 & \adjustbox{max height=0.34in, max width = 0.8in}\BEfivefour & \adjustbox{max height=0.34in, max width = 0.8in}\BEfivefour & \adjustbox{max height=0.34in, max width = 0.8in}\BEfivefour & \adjustbox{max height=0.34in, max width = 0.8in}\BEsixthree \\ 
  \hline
11&\adjustbox{max height=0.34in, max width = 0.8in}\BEfivefive  &\adjustbox{max height=0.34in, max width = 0.8in}\BEfivefive  & \adjustbox{max height=0.34in, max width = 0.8in}\BEfivefive  & \adjustbox{max height=0.34in, max width = 0.8in}\BEsixfour \\
 \hline
$2s \geq 12$& \adjustbox{max height=0.34in, max width = 0.8in}\BEsixfive & \adjustbox{max height=0.34in, max width = 0.8in}\BEsixfive & \adjustbox{max height=0.34in, max width = 0.8in}\BEsixfive & \adjustbox{max height=0.34in, max width = 0.8in}\BEsixfive
\\
  & $\mathcal{G}(n;s,s-1)$ & $\mathcal{G}(n;s,s-1)$ & $\mathcal{G}(n;s,s-1)$ & $\mathcal{G}(n;s,s-1)$ 
\\ \hline
$2s+1 \geq 13$ &  \adjustbox{max height=0.34in, max width = 0.8in}\BEsixsix & \adjustbox{max height=0.34in, max width = 0.8in}\BEsixsix & \adjustbox{max height=0.34in, max width = 0.8in}\BEsixsix & \adjustbox{max height=0.34in, max width = 0.8in}\BEsixsix \\
 & $\mathcal{G}(n;s,s)$ & $\mathcal{G}(n;s,s)$ & $\mathcal{G}(n;s,s)$ & $\mathcal{G}(n;s,s)$ \\ \hline
\end{tabular}
 \caption{Constructions for $\mathfrak{R}_q(p)$. Gray indicates density $1/2$ and green indicates density $1$. }\label{big-table}
 \end{table}

\section{Counterexamples}\label{section-counter}

Assume a graph in $\mathcal{G}(n;s,t)$ (asymptotically) achieves the maximum of $\rt(n,\#K_q,K_p,\alpha n)$ with $p\geq q+2$: it must have
 at least $q$ cells, or else the number of $K_q$ will be $o(n^q)$. In such a graph, the number of cells must always be at least the number of parts. One might expect that a $K_p$-free graph with the most copies of $K_q$ would contain copies of $K_{p-1}$, and so the sum of the number of parts and cells should be $p-1$. Subject to all these constraints, the 
 authors of \cite{balogh2017problems} conjectured that to maximize the number  of $K_q$, we should take the number of parts $t$ as large as possible. We restate their conjecture in our notation.

\begin{conj}\label{their-conjecture}
    Given integers $p>q\geq 3$, one of the asymptotically maximal graphs for $\rt(n,\#K_q,K_p,\alpha n)$ lies in $\mathcal{G}(n;q,p-q-1)$ when $p\leq 2q-1$ and lies in $\mathcal{G}(n;\lceil (p-1)/2\rceil,\lfloor (p-1)/2\rfloor)$ when $p\geq 2q$.
\end{conj}

Conjecture~\ref{their-conjecture} does not hold in general, for both $p\leq 2q-1$ and $p\geq 2q$. Already, when maximizing the number  of $K_5$, we have values of $p$ for which the  conjecture does not hold (namely, $p=10$ and $p=11$).
 For example, when $q=5$ and $p=10$, Conjecture~\ref{their-conjecture} expects $\mathcal{G}(n;5,4)$ to be optimal, but $\mathcal{G}(n;6,3)$ has more copies of $K_5$. Below, we present a more general statement, showing that the difference between the number of cells in the conjectured construction and one that contains more copies of $K_q$ can be arbitrarily large. Moreover, we show that the difference in number  of $K_q$ between the conjectured optimal and the presented construction, after normalization by $n^q$, can be arbitrary large as well. In other words, for infinitely many pairs of $p$ and $q$, we disprove the conjecture in a strong sense. However, Theorems~\ref{k4-thm} and~\ref{k5-thm} support the conjecture for $p$ large enough compared to $q$, and Theorem~\ref{general-q} shows that the conjecture does indeed hold for $p \geq 5q$. We do not believe that the constant $5$ is optimal, but it remains unclear at what point the conjecture holds or what the behavior of $\rt(n,\#K_q,K_p,\alpha n)$ should be when $p$ is small compared to $q$.

\begin{prop}
    Let $k$ be a positive integer and let $c$ be a positive real number. 
    
    \begin{enumerate}

        \item For each $q$ large enough, there exists $p\leq 2q-1$ such that an optimal graph in the conjectured family $\mathcal{G}(n;q,p-q-1)$ has $cn^q$ fewer copies of $K_q$ than some graph in $\mathcal{G}(n;q+k,p-q-k-1)$.

                \item For each $q$ large enough, there exists $p \geq 2q$ such that an optimal graph in the conjectured family $\mathcal{G}(n;\lceil\frac{p-1}{2}\rceil,\lfloor\frac{p-1}{2}\rfloor)$ has $cn^q$ fewer copies of $K_q$ than  some graph in $\mathcal{G}(n;\lceil\frac{p-1}{2}\rceil+k ,\lfloor\frac{p-1}{2}\rfloor-k)$.
    \end{enumerate}
\end{prop}

\begin{proof}
    For simplicity, we only show the proof of part 1 explicitly for odd $q$. Fix $k, c$ and let $q=2\ell+3k$ and $p = 3\ell+6k+1$ for $\ell \geq 2$ to be determined. Note that $p \leq 2q-1$.
    We shall normalize $n$ to $1$ for convenience.
    
    Consider an optimal graph (see Figure~\ref{cx-figure}) in the conjectured family $\mathcal{G}(n;q,p-q-1)$. It is easy to see that all cells should be the same size in such a graph. So it contains $\left(\frac1q\right)^q\left(\frac12\right)^\ell$ copies of $K_q$. On the other hand, a graph in $\mathcal{G}(n;q+k,p-q-k-1)$ with all cells of equal size has at least $\binom{k+q}{q}\left(\frac{1}{q+k}\right)^q\left(\frac12\right)^{2k+\ell}$ copies of $K_q$ (again see Figure~\ref{cx-figure}).

Combining counts and rearranging, we can see that the latter construction will contain  $cn^q$ more copies of $K_q$ if $1+c\leq \binom{k+q}{q}(\frac{q}{q+k})^q\left(\frac12\right)^{2k}$. 
The term $(\frac{q}{q+k})^q\left(\frac12\right)^{2k}$ is bounded below by a constant in $k$ and $\binom{k+q}{q}$ tends to infinity as $q$ increases.
Therefore, we eventually have $q=2\ell+3k$ such that $\binom{k+q}{q}(\frac{q}{q+k})^q\left(\frac12\right)^{2k}\geq 1+c$, as desired.

To prove part 2, fix $k,c$ and let $p=2q+1$ and $q$ to be determined. Again normalize $n$ to $1$. 
Consider an optimal graph in the conjectured family $\mathcal{G}(n;\lceil\frac{p-1}{2}\rceil,\lfloor\frac{p-1}{2}\rfloor)$. The symmetry of the construction yields that cell sizes should again be equal. So it contains $\left(\frac1q\right)^q$ copies of $K_q$.
On the other hand, a graph in $\mathcal{G}(n;\lceil\frac{p-1}{2}\rceil+k ,\lfloor\frac{p-1}{2}\rfloor-k)$ with all cells of equal size has at least $\binom{q+k}{q}\left(\frac{1}{q+k}\right)^k\left(\frac12\right)^{2k}$ copies of $K_q$. From here, the computation proceeds almost identically as in the first case.
\end{proof}

\begin{figure}
\begin{center}
\begin{tikzpicture}
    \draw[fill=ao(english)!70, rounded corners] (0, 0) rectangle (4, 8) {};

    \draw[fill=white]  (2, 1) ellipse [x radius=1.5, y radius=0.4];
    \draw[fill=white]  (2, 2) ellipse [x radius=1.5, y radius=0.4];
    \draw[fill=white]  (2, 4) ellipse [x radius=1.5, y radius=0.4];

    \draw (1.4,1) ellipse [x radius = 0.3, y radius = 0.3];
    \draw (2.6,1) ellipse [x radius = 0.3, y radius = 0.3];
    \draw[dashed] (1.8, 1) -- (2.2, 1);
    \draw[dashed] (1.8, 1.1) -- (2.2, 1.1);
    \draw[dashed] (1.7, 1.2) -- (2.3, 1.2);
    \draw[dashed] (1.8, 0.9) -- (2.2, 0.9);
    \draw[dashed] (1.7, 0.8) -- (2.3, 0.8);

    \draw (1.4,2) ellipse [x radius = 0.3, y radius = 0.3];
    \draw (2.6,2) ellipse [x radius = 0.3, y radius = 0.3];
    \draw[dashed] (1.8, 2) -- (2.2, 2);
    \draw[dashed] (1.8, 2.1) -- (2.2, 2.1);
    \draw[dashed] (1.7, 2.2) -- (2.3, 2.2);
    \draw[dashed] (1.8, 1.9) -- (2.2, 1.9);
    \draw[dashed] (1.7, 1.8) -- (2.3, 1.8);

    \draw (1.4,4) ellipse [x radius = 0.3, y radius = 0.3];
    \draw (2.6,4) ellipse [x radius = 0.3, y radius = 0.3];
    \draw[dashed] (1.8, 4) -- (2.2, 4);
    \draw[dashed] (1.8, 4.1) -- (2.2, 4.1);
    \draw[dashed] (1.7, 4.2) -- (2.3, 4.2);
    \draw[dashed] (1.8, 3.9) -- (2.2, 3.9);
    \draw[dashed] (1.7, 3.8) -- (2.3, 3.8);

    \draw[fill=white] (2,2.7) ellipse [x radius = 0.1, y radius = 0.1];
    \draw[fill=white] (2,3.0) ellipse [x radius = 0.1, y radius = 0.1];
    \draw[fill=white] (2,3.3) ellipse [x radius = 0.1, y radius = 0.1];

    \draw[fill=white] (2,5.2) ellipse [x radius = 0.4, y radius = 0.4];
    
    \draw[fill=white] (2,6.2) ellipse [x radius = 0.4, y radius = 0.4];
    \draw[fill=white] (2,7.2) ellipse [x radius = 0.4, y radius = 0.4];

    \draw[fill=ao(english)!70, rounded corners] (7, 0) rectangle (11, 8) {};
    
    \draw[fill=white]  (9, 1) ellipse [x radius=1.5, y radius=0.4];
    \draw[fill=white]  (9, 2) ellipse [x radius=1.5, y radius=0.4];
    \draw[fill=white]  (9, 4) ellipse [x radius=1.5, y radius=0.4];

    \draw[fill=white]  (9, 5.2) ellipse [x radius=1.5, y radius=0.4];
    \draw[fill=white]  (9, 6.2) ellipse [x radius=1.5, y radius=0.4];

    \draw[fill=white] (9,2.7) ellipse [x radius = 0.1, y radius = 0.1];
    \draw[fill=white] (9,3.0) ellipse [x radius = 0.1, y radius = 0.1];
    \draw[fill=white] (9,3.3) ellipse [x radius = 0.1, y radius = 0.1];

    \draw (8.4,1) ellipse [x radius = 0.3, y radius = 0.3];
    \draw (9.6,1) ellipse [x radius = 0.3, y radius = 0.3];
    \draw[dashed] (8.8, 1) -- (9.2, 1);
    \draw[dashed] (8.8, 1.1) -- (9.2, 1.1);
    \draw[dashed] (8.7, 1.2) -- (9.3, 1.2);
    \draw[dashed] (8.8, 0.9) -- (9.2, 0.9);
    \draw[dashed] (8.7, 0.8) -- (9.3, 0.8);

    \draw (8.4,2) ellipse [x radius = 0.3, y radius = 0.3];
    \draw (9.6,2) ellipse [x radius = 0.3, y radius = 0.3];
    \draw[dashed] (8.8, 2) -- (9.2, 2);
    \draw[dashed] (8.8, 2.1) -- (9.2, 2.1);
    \draw[dashed] (8.7, 2.2) -- (9.3, 2.2);
    \draw[dashed] (8.8, 1.9) -- (9.2, 1.9);
    \draw[dashed] (8.7, 1.8) -- (9.3, 1.8);

    \draw (8.4,4) ellipse [x radius = 0.3, y radius = 0.3];
    \draw (9.6,4) ellipse [x radius = 0.3, y radius = 0.3];
    \draw[dashed] (8.8, 4) -- (9.2, 4);
    \draw[dashed] (8.8, 4.1) -- (9.2, 4.1);
    \draw[dashed] (8.7, 4.2) -- (9.3, 4.2);
    \draw[dashed] (8.8, 3.9) -- (9.2, 3.9);
    \draw[dashed] (8.7, 3.8) -- (9.3, 3.8);

    \draw (8.4,5.2) ellipse [x radius = 0.3, y radius = 0.3];
    \draw (9.6,5.2) ellipse [x radius = 0.3, y radius = 0.3];
    \draw[dashed] (8.8, 5.2) -- (9.2, 5.2);
    \draw[dashed] (8.8, 5.3) -- (9.2, 5.3);
    \draw[dashed] (8.7, 5.4) -- (9.3, 5.4);
    \draw[dashed] (8.8, 5.1) -- (9.2, 5.1);
    \draw[dashed] (8.7, 5.0) -- (9.3, 5.0);

    \draw (8.4,6.2) ellipse [x radius = 0.3, y radius = 0.3];
    \draw (9.6,6.2) ellipse [x radius = 0.3, y radius = 0.3];
    \draw[dashed] (8.8, 6.2) -- (9.2, 6.2);
    \draw[dashed] (8.8, 6.3) -- (9.2, 6.3);
    \draw[dashed] (8.7, 6.4) -- (9.3, 6.4);
    \draw[dashed] (8.8, 6.1) -- (9.2, 6.1);
    \draw[dashed] (8.7, 6.0) -- (9.3, 6.0);

    \draw [decorate,decoration={brace,amplitude=10pt},xshift=0pt,yshift=0pt]
    (-0.1,0.8) -- (-0.1,4.2) node [black,midway,xshift=-0.7cm,yshift=0cm] 
    {\footnotesize $\ell$};

    \draw [decorate,decoration={brace,amplitude=10pt},xshift=0pt,yshift=0pt]
    (-0.1,5.0) -- (-0.1,7.4) node [black,midway,xshift=-0.7cm,yshift=0cm] 
    {\footnotesize $3k$};

    \draw [decorate,decoration={brace,amplitude=10pt},xshift=0pt,yshift=0pt]
    (6.9,0.8) -- (6.9,6.4) node [black,midway,xshift=-1cm,yshift=0cm] 
    {\footnotesize $\ell+2k$};
\end{tikzpicture}
\end{center}
\caption{Left, a graph in $\be(n;q,p-q-1)$. Right, a graph in $\be(n;q+k,p-q-k-1)$. For certain values of $p$ and $q$, the right contains more copies of $K_q.$}\label{cx-figure}
\end{figure}
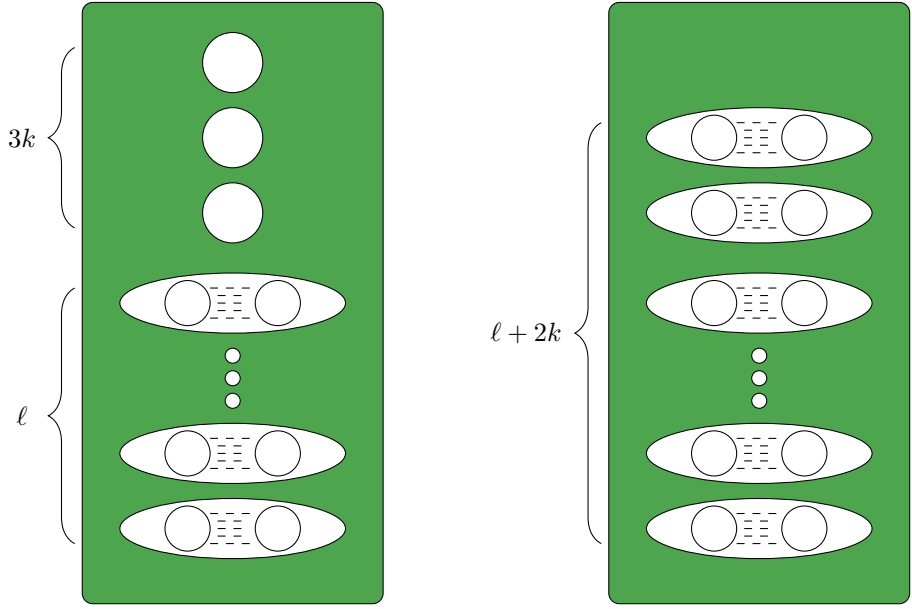

While Conjecture~\ref{their-conjecture}  makes no claims to this effect, we note here that other asymptotically maximal graphs exist for infinitely many pairs $p,q$. We loosely describe an example here (and provide a figure) for the case of $(p,q)=(6,3)$.

Let $G$ be a graph on $n$ vertices constructed as follows: First, split the vertices of $G$ into six equal parts: $V_1,V_2,\ldots,V_6$. Embed a $3$-Bollob\'as-Erd\H os graph on $V_1\cup V_2\cup V_3$, as well as on $V_4\cup V_5\cup V_6$. Add all edges between $V_1\cup V_2\cup V_3$ and $V_4\cup V_5\cup V_6$. Finally, delete all edges between $V_1$ and $V_2$, between $V_4$ and $V_5$, and between $V_3$ and $V_6$. It is reasonably straightforward to check that $G$ is $K_6$-free. Indeed, no clique contains more than three vertices among either of $V_1\cup V_2\cup V_3$ or $V_4\cup V_5\cup V_6$, or more than one vertex in $V_3\cup V_6$. Any clique containing three vertices among $V_1\cup V_2\cup V_3$ necessarily contains a vertex from $V_3$ and the symmetric condition is true for $V_4\cup V_5\cup V_6$ and a vertex in $V_6$. Putting this together, a clique containing five vertices must include three vertices from one of either $V_1\cup V_2\cup V_3$ or $V_4\cup V_5\cup V_6$, and can then contain exactly two from the other. Moreover, it is not hard to check that $G$ has an asymptotically equal number of copies of $K_3$ to a maximal graph in $\mathcal{G}(n;3,2)$.

\begin{figure}[H]
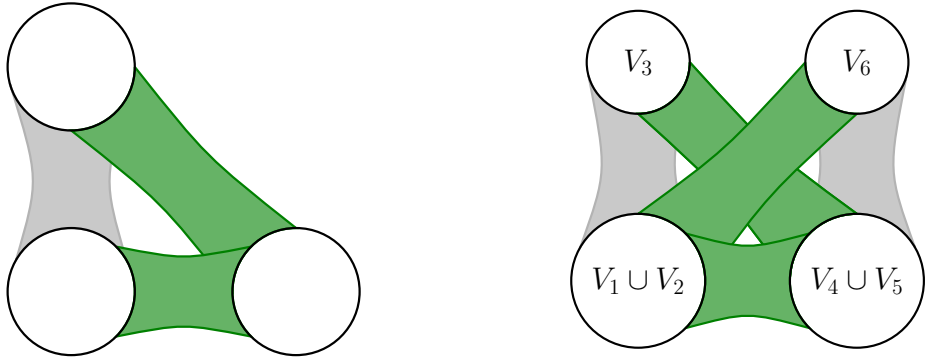

\centering
\begin{subfigure}[h]{0.4\linewidth}
\centering
\resizebox{5cm}{!}{
\conjec
}

\end{subfigure}
\hspace{1.5em}
\begin{subfigure}[h]{0.4\linewidth}
\centering
\resizebox{5cm}{!}{
\conster
}

\end{subfigure}%
\caption{A triangle-maximal graph in $\mathcal{G}(n;3,2)$ versus a non-isomorphic graph with (asymptotically) the same number of triangles.}
\end{figure}

\section{Weighted Zykov's Theorem}\label{section-zykov}

In this section we prove our main tool which will be a weighted version of Zykov's Theorem for maximizing the number of copies of $K_q$ in a $K_p$-free graph.

\begin{thm}[Zykov, \cite{zykov}]\label{zk}
Fix $2\le q< p$ and let $H$ be an $n$-vertex graph containing   no $K_p$. If $H$ has the maximum number of copies of $K_q$, then $H$ is the Tur\'an graph $T(n,p-1)$.
\end{thm}

Assign \emph{weights} $1$ and $1/2$ to the edges of a graph $R$. It will often be convenient to think of non-edges as edges of weight $0$. We call $R$ a \emph{weighted graph} and record the weight on an edge $e$ as $w(e)$.

A {\it $p$-skeleton} is a pair of vertex sets $(X,Y)$, such that $X \subseteq Y$ and every edge in $X$ has weight $1$, every edge in $Y$ has weight at least $1/2$, and $|X|+|Y| =p$.
This notion was introduced in \cite{EHSSz} as a {\it weighted-$p$-clique}. We avoid this terminology as it will create confusion with our notation for the weight of a clique.
In \cite{EHSSz} the authors were counting edges only so no such conflict was present.
A weighted graph that contains no $p$-skeleton is \emph{$p$-skeleton-free}.

In order to prove Theorem~\ref{edge-thm}, Erd\H os, Hajnal, S\'os, Szemer\'edi~\cite{EHSSz} proved the following generalization of Tur\'an's theorem that served as their main tool to analyze the structure of the reduced graph after application of the Regularity Lemma.

\begin{thm}[Weighted Tur\'an's Theorem \cite{EHSSz}]\label{w-turan}
Fix $p \geq 3$. If $R$ is an $n$-vertex weighted graph that is $p$-skeleton-free, then 
\[
\sum_{e \in E(R)} w(e) \leq    \begin{cases} 
       (1+o(1))\left(\dfrac{1}{2} \cdot \dfrac{p-3}{p-2} \right) n^2 & p \textrm{ odd,} \\
      (1+o(1))\left(\dfrac{1}{2} \cdot \dfrac{3p-10}{3p-4} \right) n^2 & p \textrm{ even.}
      
   \end{cases}
\]
\end{thm}

We will need to prove what is effectively a common generalization of these two theorems. Fortunately, both are proved via symmetrization which will also be our proof technique.

The {\it weight} of a subset of vertices $K$ in a weighted graph $R$ is the product of the edge weights on the induced edges on $K$, i.e., $w(K) := \prod_{e \in E(K)} w(e)$.
We enumerate the total weight of all copies of $K_q$ in $R$ by
\[
\N_q(R) : = \sum_{K \in \binom{V(R)}{q}} w(K).
\]
Throughout this paper, when refer to the number of copies of $K_q$ in a weighted graph $R$ we mean the value $\N_q(R)$.

\begin{defn}
Fix positive integers $t, s_1,s_2,\dots, s_t$.
    A \emph{profile graph} $R$ is defined as follows. Begin with a complete $t$-partite graph with all edges of weight $1$ and, for each $1 \leq i \leq t$, embed a balanced $s_i$-partite graph with all edges of weight $1/2$ into part $i$.
    The \emph{profile} of $R$ is the tuple $(s_1,s_2,\dots,s_t)$.
 \end{defn}

\begin{figure}[H]
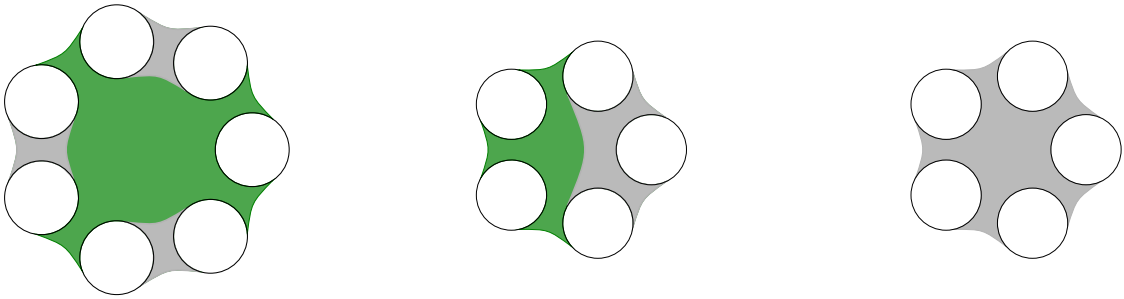

\centering
\begin{subfigure}[h]{0.3\linewidth}
\centering
\resizebox{4cm}{!}{
\BEsevenfour
}
\end{subfigure}
\hfill
\begin{subfigure}[h]{0.3\linewidth}
\centering
\resizebox{3cm}{!}{
\BEfiveprofile
}
\end{subfigure}
\hfill
\begin{subfigure}[h]{0.3\linewidth}
\centering
\resizebox{3cm}{!}{
\BEfiveone
}
\end{subfigure}%
\caption{Three profiles: $(2,2,2,1)$, $(3,1,1)$, and $(5)$.}
\end{figure}

Note that in the definition of a profile graph the initial $t$-partite graph is not necessarily balanced. Moreover, both $t=1$ and $s_i=1$ are allowed which each yield $1$-partite graphs, i.e., a set of independent vertices. Graphs in $\mathcal{G}(n;s,t)$ can be simulated by appropriate  profile graphs.

We now state our main tool. It will be proved by a two-step symmetrization argument. Note that uniqueness is not claimed in the statement.

 \begin{thm}[Weighted Zykov's Theorem]\label{w-zykov}
 Fix $2\le q< p$.
 Among $n$-vertex $p$-skeleton-free weighted graphs $R$ with $\N_q(R)$ maximum, there is a profile graph.
\end{thm}

\begin{proof}
Let $R$ be an $n$-vertex $p$-skeleton-free weighted graph with $\N_q(R)$ maximum. We will transform $R$ into a $p$-skeleton-free profile graph without decreasing $\N_q(R)$. But first, we will show that we can find a graph $R$ with the following two properties:

Say $R$ is \textit{cellular} if, whenever $x$ and $y$ are nonadjacent vertices in $R$, and $z$ is another vertex, we have $w(xz)=w(yz)$. Observe that when viewing a simple graph as a weighted graph with all edge weights $0$ or $1$, cellular is exactly the notion of complete multipartiteness. Next, define a \textit{$(\frac12,\frac12,1)$-triangle} as vertices $x,y,z$ such that $w(xy)=1/2$, $w(xz)=1/2$, and $w(yz)=1$.

We will first show that we can find a $p$-skeleton-free graph $R$ maximizing $\N_q(R)$ that is cellular and $(\frac12,\frac12,1)$-triangle-free. 
To do so, say the $q$-{\it weight} of a vertex $v$ is the sum of the weights of the $q$-cliques incident to $v$, i.e., denote
\[
w_q(v) := \sum_{v \in K \in \binom{V(R)}{q}} w(K). 
\]
Let $x,y$ be distinct non-adjacent vertices with $w_q(x) \geq w_q(y)$. Let us {\it symmetrize} $y$ to $x$, i.e., for each $z \in V(R) \setminus \{x,y\}$ replace the edge weight on $yz$ with that on $xz$. This operation does not decrease the weight $\N_q(R)$ as $w_q(x) +w_q(y) + \N_q(R-x-y) \leq 2w_q(x) + \N_q(R-x-y)$. There is no $p$-skeleton containing $x$, so after symmetrizing there can be no $p$-skeleton containing $y$, therefore $R$ remains $p$-skeleton-free.

We repeatedly apply this operation in the following manner. Fix an arbitrary order
 $v_1,v_2,\ldots,v_n$ of the vertices of $R$. Let $U$ be the set of these vertices, and begin with $S=\emptyset$. While $U$ is nonempty, proceed as follows:

Take a vertex with largest $q$-weight in $U$, say $v_i$, and move it to $S$. For each $v_j\in U$ non-adjacent to $v_i$, symmetrize $v_j$ to $v_i$, then move $v_j$ from $U$ to $S$. Iterating these two steps, repeatedly moving a largest $q$-weight vertex from $U$ to $S$ and symmetrizing its non-neighbors to it, yields a cellular graph.

Now, suppose we have vertices $x,y,z$ such that $w(xz)=w(yz)=1/2$ but $ w(xy)=1$. We show that we can transform our graph in a way that reduces the number of such triples while maintaining each of our other desired properties---namely, $\N_q(R)$ maximal, $p$-skeleton-freeness and cellularity. Before we do so, suppose $R$ has $s$ classes.
Observe that $R$ does not contain a $(p-s)$-clique whose edges are all $1$. Otherwise take these vertices as $X$. Adding one vertex from each of the remaining classes to form $Y$ yields a $p$-skeleton $(X,Y)$, a contradiction. A single vertex is a $1$-clique whose edges are all $1$, so it follows that $s<p-1$.

Now we perform a second transformation step. Let $x,y, z$ be vertices in different classes of $R$ such that $w(xy)=w(yz)=1/2$, $xz=1$, i.e., they form a $\{1/2,1/2,1\}$-triangle. Let $R_x$ denote the graph yielded from $R$ by, whenever we have a vertex $v$ such that $w(xv)>0,w(yv)>0$, replace the weight on $yv$ with the weight on $xv$. As $R$ does not contain a $(p-s)$-clique of all edge-weights $1$, then neither will $R_x$. Define $R_y$ similarly, and see that neither $R_x$ nor $R_y$ contains a $p$-skeleton.

We claim either $\N_q(R_x) \geq \N_q(R)$ or $\N_q(R_y) \geq \N_q(R)$.
Suppose to the contrary that $\N_q(R_x) < \N_q(R)$ and $\N_q(R_y) < \N_q(R)$. We introduce the following notation: for a vertex $v\in V(R)$ and set of vertices $S\subseteq V(R)$, let $\pi_S(v):=\prod_{u\in S}w(vu)$. Notably, if $S$ is a set of vertices with weight $w(S)$, then the weight $w(S\cup \{v\})$ is  $\pi_S(v)w(S)$. 

Given our above assumption, we have
\begin{align*}
0 & > \N_q(R_x)-\N_q(R) \\
& = \sum_{S\in \binom{V(R)\setminus\{x,y\}}{q-2}}\frac{1}{2}(\pi_S(x)-\pi_S(y))\pi_S(x)w(S)+ \sum_{T\in \binom{V(R)\setminus\{x,y\}}{q-1}}(\pi_T(x)-\pi_T(y))w(T)    
\end{align*} 
and 
\begin{align*}
0&>\N_q(R_y)-\N_q(R) \\
&= \sum_{S\in \binom{V(R)\setminus\{x,y\}}{q-2}}\frac{1}{2}(\pi_S(y)-\pi_S(x))\pi_S(y)w(S)+ \sum_{T\in \binom{V(R)\setminus\{x,y\}}{q-1}}(\pi_T(y)-\pi_T(x))w(T).    
\end{align*}
    Adding these two inequalities, we see the second terms cancel and we have 
\[
0> \sum_{S\in \binom{V(R)\setminus\{x,y\}}{q-2}}\frac{1}{2}(\pi_S(x)-\pi_S(y))\pi_S(x)w(S) + \sum_{S\in \binom{V(R)\setminus\{x,y\}}{q-2}}\frac{1}{2}(\pi_S(y)-\pi_S(x))\pi_S(y)w(S).
\]
Combining these sums and rearranging yields
\[
0>\frac{1}{2} \sum_{S\in \binom{V(R)\setminus\{x,y\}}{q-2}} (\pi_S(y)-\pi_S(x))^2 w(S), 
\]
implying that a linear combination of squares with non-negative coefficients is negative, a contradiction. 

So, $\N_q(R_x)+\N_q(R_y)-2\N_q(R)\geq 0$. If $\N_q(R_x)-\N_q(R)<0$, then $\N_q(R_y)-\N_q(R)>0$, which would contradict our choice of $R$ maximizing $\N_q(R)$. Therefore, we must have $\N_q(R_x)=\N_q(R_y)=\N_q(R)$. Importantly, this means we can transform our graph from $R$ to either $R_x$ or $R_y$ without a decreasing $\N_q$ and remaining $p$-skeleton-free.

Now, without loss of generality, assume $x$ is in at least as many $(\frac12,\frac12,1)$-triangles  as $y$ is. Then $R_y$ has fewer copies of such triangles  than $R$: recall that $x$ and $y$ were chosen along with a $z$ such that $x,y,z$ formed such a triangle that has been deleted. Moreover, each vertex that was nonadjacent to $x$ in $R$ remains nonadjacent in $R_y$, but will have a different neighborhood than $x$. Symmetrizing each such vertex to $x$ recovers the cellular property. After doing so, $R$ is $p$-skeleton-free with $\N_q(R)$ maximum and has fewer $(\frac12,\frac12,1)$-triangles than before.
We continue this process until $R$ has no $(\frac12,\frac12,1)$-triangles. 

In a cellular graph, nonadjacency is an equivalence relation. So the graph admits a natural partition into classes that are maximal independent sets---call these \emph{cells}.
It is easily verified that, if a cellular graph is $(\frac12,\frac12,1)$-triangle-free, it admits another natural equivalence relation, where $x\sim y$ if $w(xy)\leq 1/2$. Call these equivalence classes \emph{parts}. Note that the partition into cells is a refinement of the partition into parts.  Vertices in different parts are joined by an edge of weight $1$. Vertices in the same part but different cells are joined by an edge of weight $1/2$. And vertices in the same cell are nonadjacent.

Now we have that $R$ is $p$-skeleton-free, maximizes $\N_q(R)$, and is both cellular and $(\frac12,\frac12,1)$-triangle-free. It remains to show that we may suppose two cells belonging to the same part must be of equal-as-possible size.
If not, replacing the neighborhood of a vertex $y$ in the larger cell with the neighborhood of a vertex $x$ in the smaller cell increases edges between the cells, without altering the number of $K_q$ using one or fewer vertices from these two cells. 
 With this we may conclude that $R$ can be transformed into a profile graph that is $p$-skeleton-free with $\N_q(R)$ maximum.
\end{proof}

\section{Proofs}\label{section-proofs}

In this section we prove our main Theorems~\ref{k4-thm}, \ref{k5-thm}, \ref{q-plus}, and \ref{general-q}. Along the way we reprove Theorems~\ref{edge-thm} and \ref{triangle-thm}. Each of these proofs use the same general approach. Beginning with an extremal graph $G$, we build a weighted reduced graph $R$ via the Regularity Lemma. Then by the Weighted Zykov Theorem we show $R$ can be assumed to be a profile graph. Through a series of intermediate lemmas we show that we can assume that the profile of $R$ satisfies a certain structure. Given this, we can construct a graph $H$ from $\mathcal{G}(n;s,t)$ that has the same (asymptotically) number of  $K_q$ as $G$. As we understand the bounded structure of $H$, we can estimate $\N_q(G)$ by $\N_q(H)$ and thus determine $\mathfrak{R}_q(p)$.

We include the needed standard definitions to state the Szemer\'edi Regularity Lemma~\cite{SZRL} (see~\cite{reg-survey} for a survey). A pair of vertex classes $(U,V)$ have \emph{density} $d(U,V) = \frac{e(U,V)}{|U||V|}$ where $e(U,V)$ counts the edges between $U$ and $V$. For $\varepsilon>0$, the pair $(U,V)$ is \emph{$\varepsilon$-regular} if for every $U' \subseteq U$ and $V'\subseteq V$ with $|U'| \geq \varepsilon|U|$ and $|V'| \geq \varepsilon|V|$ we have $|d(U,V)-d(U',V')| < \varepsilon$.

\begin{lemma}[Szemer\'edi Regularity Lemma~\cite{SZRL}]
    For $\varepsilon >0$, there exists an $M = M(\varepsilon)$ such that, for every graph $G$, there is partition $V_1,V_2,\dots, V_r$ of the vertices of $G$ into $r$ classes such that
    \begin{itemize}
        \item $1/\varepsilon \leq r \leq M$,
        \item $||V_i|-|V_j|| \leq 1$ for all $i,j$,
        \item $(V_i,V_j)$ is $\varepsilon$-regular for all but at most $\varepsilon r^2$ of the pairs $(i,j)$.
    \end{itemize}
\end{lemma}

\smallskip

\noindent Let $q \geq 2$ and $p \geq q+2$. Fix constants from right to left satisfying 
\[
0 < \alpha \ll \varepsilon \ll \delta \ll \gamma < 1.
\]
Suppose $n$ is sufficiently large and let $G$ be an extremal graph for $\rt(n,\#K_q,K_p,\alpha n)$. That is, $G$ is $n$-vertex $K_p$-free with independence number at most $\alpha n$ and the maximum number of $K_q$ copies.
We will show that there is a graph $H$ in $\mathcal{G}(n;s,t)$ (for $s$ and $t$ to be determined later) such that $H$ is $K_p$-free and has independence number at most $\alpha n$ and
\[
\N_q(H) \leq \N_q(G) \leq \N_q(H) +  4 \gamma n^q.
\]
Apply the Regularity Lemma with $\varepsilon$ to $G$ to obtain an 
equipartition $V_1,V_2,\dots, V_r$ into $r$ parts such that $1 /\varepsilon \leq r \leq M=M(\varepsilon)$. 
Throughout this section we will ignore floors and ceilings on $\lfloor \frac{n}{r} \rfloor \leq |V_i| \leq \lceil \frac{n}{r} \rceil$ and assume $|V_i| = \frac{n}{r}$ for all $i$.

A usual ``cleanup'' argument shows that there are at most $\varepsilon n^q$ copies of $K_q$ using two vertices in a class $V_i$, at most $\varepsilon n^q$ copies of $K_q$ using an edge between $V_i,V_j$ that are not $\varepsilon$-regular, and at most $\delta n^q$ copies of $K_q$ using an edge between $V_i,V_j$ that have density less than $\delta$.

Build a weighted reduced graph $R$ on $r$ vertices $\{v_1,v_2,\dots, v_r\}$ as follows.

\begin{itemize}
    \item $v_iv_j$ is an edge of weight $1$ if $(V_i,V_j)$ is $\varepsilon$-regular and has density at least $1/2+\delta$. 
    
    \item $v_iv_j$ is an edge of weight $1/2$ if $(V_i,V_j)$ is  $\varepsilon$-regular  and has density at least $\delta$.
    
    \item $v_iv_j$ is a non-edge otherwise.
\end{itemize}

That $R$ contains no $p$-skeleton $(X,Y)$, follows from a lemma implicit in \cite{EHSSz} (see also \cite{balogh2017problems}). It can be proved via a standard embedding argument with $\varepsilon$-regular pairs of sufficient density and sublinear independence number.

\begin{lemma}(Erd\H os, Hajnal, S\'os, Szemer\'edi, \cite{EHSSz})
    For every  $\varepsilon > 0$, and integer $k$, there exists $\alpha>0$ such that for every $n$-vertex graph $G$   with independence number at most $\alpha n$, if the weighted reduced graph $R$ contains a $k$-skeleton, then $G$ contains a clique $K_k$, for every $n$ sufficiently large. 
\end{lemma}

This immediately implies that $R$ is $p$-skeleton-free. Next we establish a correspondence between copies of $K_q$ in $R$ and $G$ by a standard counting lemma.

\begin{lemma}[Counting Lemma]\label{counting-lemmma}
    Fix $\varepsilon >0$.
    Let $V_1,V_2,\dots, V_q$ be a set of $q$ vertex classes of $G$ such that each pair $(V_i,V_j)$ is $\varepsilon$-regular and has density $d(V_i,V_j)$.
    Then the number of copies of $K_q$ is at most
    \[
    \left(\prod_{i<j} d(V_i,V_j)  + \varepsilon \binom{q}{2}\right)\left(\frac{n}{r}\right)^q.
    \]    
\end{lemma}

Let $K$ be a copy of a $K_q$ in the reduced graph $R$. 
An edge of weight $1/2$ in $K$ corresponds to a pair of classes in $G$ of density less than $1/2+\delta$ and an edge of weight $1$ in $K$ corresponds to a pair of classes in $G$ of density at most $1$. Therefore, by Lemma~\ref{counting-lemmma}, the clique $K$ corresponds to at most 
\[
\left(w(K) + \gamma \right) \left(\frac{n}{r}\right)^q
\]
copies of $K_q$ in $G$,  so the  number of $K_q$ copies in $G$ can be upper bounded
\[
\N_q(G) \leq \left(\N_q(R) +\binom{r}{q} \gamma \right) \left(\frac{n}{r}\right)^q + (2\varepsilon+\delta)n^q \leq \N_q(R)  \left(\frac{n}{r}\right)^q + 2\gamma n^q.
\]

In the following we will establish a number of claims that allow us to adopt (without loss of generality) assumptions on the structure of $R$ by transforming $R$ in such a way that $\N_q(R)$ does not decrease by more than $\gamma r^q$.
Therefore, after adopting these transformations, $R$ will satisfy
 the inequality 
 \begin{equation}\label{upper-bound-ineq}
 \N_q(G) \leq  \N_q(R)  \left(\frac{n}{r}\right)^q + 3\gamma n^q.    
 \end{equation}
By the Weighted Zykov Theorem (Theorem~\ref{w-zykov}), we may suppose $R$ is a profile graph with $t$ parts and profile $(s_1,s_2,\dots, s_t)$. Let $s:= \sum_{i=1}^t s_i$ be the total number of cells in $R$. Clearly, if $\N_q(R)>0$ we must have the trivial bounds $s \geq q$ and $s \geq t \geq 1$ and as $R$ contains no $p$-skeleton, we have $t+s\leq p-1$.
As $R$ is a profile graph, the cells in a part form a complete balanced multipartite graph. 
Therefore, two cells in part differ in cardinality by at most $1$.
For simplicity of computation, we will make the additional assumption that the cells in a part are of exactly the same. Indeed, removing a single vertex from $R$ destroys at most $\varepsilon r^q$ copies of $K_q$ and as there are at most $p-2$ cells in $R$, we remove vertices from $R$ until each pair of cells in the same part have the same cardinally. at the cost of $(p-2)\varepsilon  r^q \ll \delta r^q$ copies of $K_q$.


\subsection{Profile Lemmas}

In this subsection we prove lemmas that allow us to control the structure of the profile of $R$, e.g., the number of cells per part.

\begin{claim}\label{part-plus-cell}
    We may suppose $s+t=p-1$. 
\end{claim}

\begin{proof}
As $R$ contains no $p$-skeleton, we have $t+s\leq p-1$. So suppose $t+s< p-1$. 
Without loss of generality, we may suppose that part $1$ contains $\Omega(r)$ vertices. 
Then take a cell from part $1$ and partition it into cells $S$ and $S'$ of (almost) equal size and add all edges of weight $1/2$ between $S$ and $S'$. This defines a new profile $(s_1+1,s_2,\dots, s_t)$.
Clearly, this transformation does not decrease $\N_q(R)$ and the resulting graph is $p$-skeleton-free and has one more cell.
We may repeat this transformation until $t+s=p-1$.
\end{proof}



\begin{claim}[Relative cell size]\label{relative-cell}
     Let $T,T'$ be parts containing $a$ and $b$ cells, respectively, with $x$ the number of vertices in a cell in $T$ and $y$ the number of vertices in a cell in $T'$. If $a \geq b$, then we may suppose $y\geq x$, and thus if $a=b$, we may suppose $x=y$.
\end{claim}

\begin{proof}
    Let $x$ denote the size of a cell $S$ in $T$ and $y$ the size of a cell  $S'$ in $T'$. We will show that if $a\geq b$, we can transform our graph until $y \geq x$ without decreasing the number  of $K_q$. 

    Suppose that $x>y$. We transform $R$ by moving a vertex $u$ from $S$ to $S'$, i.e., deleting all edges on $u$ and replacing its neighborhood with that of a vertex in $S'$. We will show that for each $\ell$ with $0\leq\ell\leq q-1$, the number  of $K_{\ell+1}$ spanned by $T$ and $T'$ does not decrease.

    Clearly, the number  of $K_{\ell+1}$ not including $u$ is unchanged. Observe further that, as $x>y$, $u$ is in at least as many edges using two vertices from $S$ and $S'$. Prior to this transformation, $u$ was adjacent to the $y$ vertices of $S'$, and after this transformation, $u$ is adjacent to the $x-1\geq y$ remaining vertices of $S$. So, if one fixes a copy of $K_{\ell-1}$ using no vertices of $S$ or $S'$, the number of $K_{\ell+1}$ containing it does not decrease.

    It then remains to show that the number  of $K_{\ell+1}$ using $u$ and no other vertex from $S$ or $S'$ does not decrease. We will consider copies of $K_\ell$ using no vertex in $S$ or $S'$. Suppose such a copy uses $j$ vertices from $T\setminus S$ and $\ell-j$ vertices from $T'\setminus S'$. If $j=\ell-j$, our transformation does not change the weight of the copy of $K_{\ell+1}$ given by this clique together with $u$. If $j>\ell-j$, the corresponding copy of $K_{\ell+1}$ will at least double in weight, and if $j<\ell-j$, the copy of $K_{\ell+1}$ including it and $u$ will have its weight decrease to a smaller, yet still positive, value. So, if there are at least as many copies of $K_\ell$ using more vertices from $T\setminus S$ than from $T'\setminus S'$, this transformation will at least double the contribution of at least half the copies of $K_\ell$, thus the total number of copies will not decrease.

    Fixing $j>\ell-j$, one can easily see that the number of $K_\ell$ using $j$ vertices from $T\setminus S$ and $\ell-j$ from $T'\setminus S'$ is $\binom{a-1}{j}x^j\binom{b-1}{\ell-j}y^{\ell-j}$. Similarly, one can find the number  of $K_\ell$ using $\ell-j$ vertices from $T\setminus S$ and $j$ from $T'\setminus S'$ is given by $\binom{a-1}{\ell-j}x^{\ell-j}\binom{b-1}{j}y^j$.

Dividing both counts by $x^{\ell-j}y^j$ and using $x \geq y$ gives a  difference of
\begin{align*}
\left(\frac{x}{y}\right)^{2j-\ell} & \left[ \binom{a-1}{j}\binom{b-1}{\ell-j}  - \binom{a-1}{\ell-j}\binom{b-1}{j}\right] 
\geq \binom{a-1}{j}\binom{b-1}{\ell-j}- \binom{a-1}{\ell-j}\binom{b-1}{j} \\ \smallskip
& =  \frac{(a-1)!(b-1)!}{j!(\ell-j)!(a-1-\ell+j)!(b-1-\ell+j)!}\left[
\frac{(a-1-\ell+j)!}{(a-1-j)!} - \frac{(b-1-\ell+j)!}{(b-1-j)!}
\right],
\end{align*}
which is non-negative whenever $a\geq b$.

With this, we may transform $R$ by repeatedly moving vertices of $T$ into cells of $T'$ without decreasing the number of copies of $K_q$ in $R$ until $y\geq x$, i.e., cells in $T'$ are at least as large as cells in $T$. 
\end{proof}

\begin{claim}[Cell Lemma, I]\label{cell-lemma1}
We may suppose that the number of cells $s_i$ in part $i$ satisfies $s_i \leq q$. 
\end{claim}

\begin{proof}
Suppose that $s_i>q$. Put $k:=s_i$ and
let $S_1, S_2,  \dots, S_k$ be the cells in part $i$ and suppose they are normalized as $|S_i|=1$ for all $1 \leq i \leq k$.
Then replace part $i$ with two parts, one with with cells $S_1,S_2,\dots, S_{k-2}$, and the other with a single cell of the remaining vertices $S_{k-1} \cup S_{k}$.

Every $K_q$ in $R$ intersects part $i$ in a copy of $K_\ell$ for some $0 \leq \ell \leq q$. For $\ell=0,1$, the number  of $K_q$ is unchanged. We will show that the number  of $K_\ell$ does not decrease locally and thus the number  of $K_q$ does not decrease, globally, in $R$.

The number of  $K_\ell$ using no vertices from $S_{k-1}\cup S_{k}$ is unchanged by this operation. The number of copies  $K_\ell$ using two vertices in $S_{k-1}\cup S_{k}$ drops to zero from $\binom{k-2}{\ell-2}(\frac{1}{2})^{\binom{\ell}{2}}$. The number of $K_\ell$ using exactly one vertex from $S_{k-1}\cup S_{k}$ increases to $2\binom{k-2}{\ell-1}(\frac{1}{2})^{\binom{\ell}{2}-(\ell-1)}$ from $2\binom{k-2}{\ell-1}(\frac{1}{2})^{\binom{\ell}{2}}$. Indeed, the weight of any edge with exactly one vertex in $S_{k-1} \cup S_k$ doubles. In such a copy of $K_\ell$ there are $\ell-1$ of these edges. Thus the net change in the number of copies of $K_\ell$ is given by 
\[
2\binom{k-2}{\ell-1}\left(\frac{1}{2}\right)^{\binom{\ell}{2}-(\ell-1)}-2\binom{k-2}{\ell-1}\left(\frac{1}{2}\right)^{\binom{\ell}{2}}-\binom{k-2}{\ell-2}\left(\frac{1}{2}\right)^{\binom{\ell}{2}}.
\]
Dividing this expression by the (positive) quantity $\binom{k-2}{\ell-1}(\frac{1}{2})^{\binom{\ell}{2}}$ yields 
\[
2^\ell -2 -\frac{\binom{k-2}{\ell-2}}{\binom{k-2}{\ell-1}} = 2^\ell -2 - \frac{\ell-1}{k-\ell} \geq 2^\ell - 2 -(\ell-1) = 2^\ell -\ell -1,
\]
which is greater than zero whenever $\ell$ is at least two.
Therefore, this operation does not decrease the number  of $K_q$ in $R$.
Repeating this argument eventually gives $s_i \leq q$.
\end{proof}

\begin{claim}[Cell Lemma, II]\label{cell-lemma2}
Denote by $k$ the maximum number of cells in a part of $R$. Then either   $k\leq 2$ or
\begin{equation}\label{k-bound}
2^{k-2}\left(\frac{k}{k-1}\right)^{k-1}-k \leq \frac{q-k+1}{s-q},    
\end{equation}
where $s > q$ denotes the total number of cells in $R$.
\end{claim}

\begin{proof}
Suppose that part $i$ has $k$ cells with $s > q \geq k \geq 3$ and (\ref{k-bound}) does not hold.
Let $S_1,S_2,\dots, S_k$ be the cells in part $i$ and define $x:=|S_i|$ for all $1\leq i \leq k$.

Perform the following transformation of part $i$ by replacing it with two parts, one with cells $S_1,S_2,\dots, S_{k-2}$, and the other with a single cell $S_{k-1} \cup S_k$. Then redistribute vertices so that each of the resulting $k-1$ cells  $S_1,  S_2, \dots,  S_{k-2}$ and  $S_{k-1} \cup S_k$ have cardinality ${kx}/{(k-1)}$.

Every $K_q$ in $R$ intersects part $i$ in a copy of $K_\ell$ for some $0 \leq \ell \leq q$. For $\ell=0,1$, the number  of $K_q$ is unchanged. We will show that the number of $K_\ell$ does not decrease locally and thus the number  of $K_q$ does not decrease, globally, in $R$.

    First let us show that the number of $\ell$-cliques does not decrease for $2 \leq \ell \leq k-2$. The number  of $K_\ell$ in part $i$ is
    \[
    \binom{k}{\ell} x^\ell \left(\frac{1}{2}\right)^{\binom{\ell}{2}}, 
    \]
while the number in the two parts we replace it with is 
\[
\binom{k-2}{\ell}x^\ell  \left(\frac{k}{k-1}\right)^\ell\left(\frac{1}{2}\right)^{\binom{\ell}{2}}+ \binom{k-2}{\ell-1}x^\ell \left(\frac{k}{k-1}\right)^\ell\left(\frac{1}{2}\right)^{\binom{\ell-1}{2}}.
\]
Their difference, after multiplying by  $x^{-\ell}2^{\binom{\ell}{2}}$, is
\begin{align*}
    \left[\binom{k-2}{\ell} + \binom{k-2}{\ell-1}2^{\ell-1}\right] \left(\frac{k}{k-1}\right)^\ell -\binom{k}{\ell}.
\end{align*}
For $\ell=2$ this gives
\[
\left[\binom{k-2}{2} +  2 (k-2)\right]\left(\frac{k}{k-1}\right)^2 - \binom{k}{2}=
 \frac{(k-2)(k+1)k^2}{2(k-1)^2}- \binom{k}{2}\geq 0 
\]
as $k \geq 3$.
For $\ell \geq 3$ we have
\begin{align*}
    \left[\binom{k-2}{\ell} + \binom{k-2}{\ell-1}2^{\ell-1}\right] \left(\frac{k}{k-1}\right)^\ell -\binom{k}{\ell}  
& \geq  \left[\binom{k-1}{\ell}  + \binom{k-2}{\ell-1}\left(2^{\ell-1}-1\right)\right] -\binom{k}{\ell}
\\
& \geq   \binom{k-2}{\ell-1}\left(2^{\ell-1}-1\right) -\binom{k-1}{\ell-1}\ge 0,
\end{align*}
 whenever $\ell \geq 3$.
Therefore, for each $\ell \leq k-2$,  the number of $K_q$ that intersect the vertices of part $i$ in $\ell$ vertices does not decrease.

Now let us count the other copies of $K_q$. For notational convenience, we use $W$ to denote the subgraph of $R$ induced on all vertices not in part $i$, this graph and the edges between it and part $i$ is unchanged by the transformation. As before, $\N_\ell(W)$ denotes the number of $K_\ell$ in $W$. Now, the number of copies of $K_q$ that use $k-1$ or $k$ vertices from part $i$ is
\[
\binom{k}{k-1} x^{k-1} \left(\frac{1}{2}\right)^{\binom{k-1}{2}}  \N_{q-k+1}(W)  +  x^{k}\left(\frac{1}{2}\right)^{\binom{k}{2}}   \N_{q-k}(W).
\]
The number of $K_q$ that use $k-1$ vertices after the transformation (there are none that use $k$ vertices) is
\[
 x^{k-1} 
\left(\frac{k}{k-1}\right)^{k-1} \left(\frac{1}{2}\right)^{\binom{k-2}{2}}\N_{q-k+1}(W). \]
The difference of those cardinalities, after multiplying by $x^{-(k-1)}2^{\binom{k}{2}}$, is
\begin{equation}\label{class-compare}
      \left[2^{2k-3} \left(\frac{k}{k-1}\right)^{k-1} - k 2^{k-1} \right]\N_{q-k+1}(W) - x \cdot \N_{q-k}(W). 
\end{equation}

Now let us double count pairs $(K,v)$, where $K$ is a $(q-k+1)$-clique in $W$ and $v$ a vertex in $K$. First,  one may count such cliques, and then the ways to mark $v$ in that clique. Second, we may count the ways to extend a clique of size $q-k$ by adding vertex $v$. We can fix a clique in one of $\N_{q-k}(W)$ ways. There are $s-k$ cells in $W$, and $q-k$ cells including a vertex from such a clique. So there are $s-q$ cells from which to select $v$, and these cells must be at least as large as those in part $i$ by Claim~\ref{relative-cell}. Each edge between vertices in these cells and those in the $(k-q)$-clique have weight at least $1/2,$ yielding 
\[
(q-k+1)  \N_{q-k+1}(W) \geq  (s-q) x \left(\frac{1}{2}\right)^{k-1}\N_{q-k}(W),
\]
which implies
\[
\left(\frac{1}{s-q}\right) 2^{k-1}(q-k+1)  \N_{q-k+1}(W) \geq  x \cdot  \N_{q-k}(W).
\]
Now, as (\ref{k-bound}) does not hold, we may use the above inequality to conclude that (\ref{class-compare}) is non-negative. So this transformation does not decrease the number of copies of $K_q$ in $R$, and we may repeat it,  until the claim holds.    
\end{proof}

Note that inequality (\ref{k-bound}) in Claim~\ref{cell-lemma2} does not hold when $k=q$ and $s>q\geq 3$ which together with Claim~\ref{cell-lemma1} implies the following useful statement.

\begin{cor}\label{cell-cor}
    If the number of cells $s$ in $R$ satisfies $s>q \geq 3$, then no part contains $q$ cells.
\end{cor}

We conclude this subsection with a lemma that estimates the change in number of small cliques when transforming two parts of two cells each into three parts of one cell each. This lemma will be used repeatedly in the proofs in the following subsections.

\begin{lemma}\label{repartition-two-two}
Suppose $R$ contains two parts each containing exactly two cells. If each of these cells has cardinality $3x$, then there is a transformation of these two parts into three parts, each containing a single cell, such that the set of transformed vertices span $3x^2$ more edges, $10x^3$ more triangles and $\frac{81}{4}x^4$ fewer copies of $K_4$.
\end{lemma}

\begin{proof}
Transform $R$ by replacing the two parts of two cells with three parts each containing a single cell of cardinality $4x$. The resulting graph has one more part and one less cell, so it remains $p$-skeleton-free.
 From here it is a  straightforward computation to confirm that the number of edges increases from  $45x^2$ to $48x^2$, the number of triangles increases from $54x^3$ to $64x^3$, and the number of $K_4$ decreases from $\frac{81}{4}x^4$ to $0$.   
\end{proof}

\subsection{Key Lemma}

In this subsection we connect $\N_q(G)$ to $\N_q(R)$ for a corresponding profile graph $R$.

\begin{lemma}[Key Lemma]\label{key-lemma}
    If $R$ has $s$ cells and the $t$ parts of $R$ each contain $\lceil s/t \rceil$ 
    or $\lfloor s/t \rfloor$ cells, then 
\[
\N_q(G) = (1+o(1))\max_x \{\N_q(R')\},
\]
where $x$ is the number of vertices in a part of order $\lceil s/t \rceil$ of graph $R'$ with the same profile as $R$. Alternatively,
\[
\N_q(G) = (1+o(1)) \max \left \{ \N_q(H) \mid H \in \mathcal{G}\left(n;s,t\right)\right \}.
\]
\end{lemma}

\begin{proof}
Let $(s_1,s_2,\dots, s_t)$ be the profile of $R$. 
Construct $H_R\in \mathcal{G}(n;s,t)$ with parts $V_1,V_2,\ldots,V_t$ such that:
    \begin{itemize}
        \item  $\displaystyle \frac{|V_i|}{n}=\frac{x_i}{r}$, where $x_i$ is the number of vertices in part $i$ of $R$.
        \item $V_i$ spans an $s_i$-Bollob\'as-Erd\H os graph.
    \end{itemize}

By construction (see Section~\ref{construct}), $H_R$ has independence number $\alpha n$ (as $n$ is large enough). Moreover, it can be seen that since $R$ is $p$-skeleton-free, $H_R$ will be $K_p$-free. Therefore, $\N_q(H_R) \leq \N_q(G)$.

Observe that each cell of $H_R$ is paired naturally with a corresponding cell in $R$ and is larger by a factor of ${n}/{r}$. Therefore, each set of $q$ cells in $H$ contains $(1+o(1))\left(\frac{n}{r}\right)^q$ times as many copies of $K_q$ as the corresponding set of $q$ cells in $R$. Summing over all $\binom{s}{q}$ sets of $q$ cells shows $\N_q(H_R)=(1+o(1))\left(\frac{n}{r}\right)^q\N_q(R)$. Thus, as $n$ is large enough, 
\[
\left|\N_q(H_R) - \N_q(R)\left(\frac{n}{r}\right)^q\right|< \varepsilon n^q.
\]
 Together with (\ref{upper-bound-ineq}) this gives
 \begin{align*}
\N_q(R)\left(\frac{n}{r} \right)^q - \varepsilon n^q \leq  \N_q(H_R) \leq \N_q(G) \leq \N_q(R) \left(\frac{n}{r}\right)^q + 3 \gamma n^{q}  \leq \N_q(H_R) + 4\gamma n^q.
 \end{align*}

Observe that, for any $R'$ with the same profile as $R$, we can construct a corresponding $H_{R'}$. So, for an $R'$ maximizing $\N_q(R')$, the preceding inequalities hold. Thus in light of Claims~\ref{part-plus-cell} and \ref{relative-cell}, the maximization of $\N_q(R')$ can be written as univariate optimization in $x$, the number of vertices in a part of size $\lceil s/t \rceil$.

To prove the alternative formulation, we can proceed in symmetric fashion, yielding $R'$ for any $H\in \mathcal{G}(n;s,t)$ with $\left|\N_q(R')\left(\frac{n}{r}\right)^q - \N_q(H)\right|< \varepsilon n^q$.
\end{proof}

The consequence of Lemma~\ref{key-lemma} is that to compute $\mathfrak{R}_q(p)$, we simply need to determine $\max_x \{\N_q(R')\}$.
 Moreover, if $t \vert s$, each part contains the same number of cells, so there is no optimization by Claim~\ref{relative-cell} and $\N_q(R)$ can be computed directly.

Whenever we can prove that the parts of $R$ satisfy $|s_i-s_j|\leq 1$ for all $i,j$, we may apply Lemma~\ref{key-lemma}. In the following subsections, the most frequent resulting profiles will be addressed by the next two lemmas.

\begin{lemma}\label{p-large}
If $s=t+1$, then $p=2s$ and
\[
\mathfrak{R}_q(2s) =  \max_{0 \leq x \leq 1}
       \frac{1}{2}\left(\frac{x}{2}\right)^2\binom{s-2}{q-2}\left(\frac{1-x}{s-2}\right)^{q-2}+x\binom{s-2}{q-1}\left(\frac{1-x}{s-2}\right)^{q-1} + \binom{s-2}{q}\left(\frac{1-x}{s-2}\right)^q.   
\]
    If $s=t$, then $p=2s+1$ and
    \[
\mathfrak{R}_q(2s+1) = \binom{s}{q} \left(\frac{1}{s}\right)^q. 
\]
Alternatively, if $s-t\leq 1$, then
\[
\N_q(G) = (1+o(1)) \max \left \{ \N_q(H) \mid H \in \mathcal{G}\left(n;\left \lceil \frac{p-1}{2}\right \rceil,\left \lfloor \frac{p-1}{2} \right \rfloor\right)\right \}.
\]
\end{lemma}

\begin{proof}
If $s=t+1$, Claim~\ref{part-plus-cell} gives $s+t=p-1$, so $p=2s$. Moreover, one part of $R$ contains exactly $2$ cells and the remaining parts each contain a single cell. Let $x$ be the number of vertices in the part with two cells. It is straightforward to count the number of $K_q$ copies in $R$ as a function of $x$. Applying Lemma~\ref{key-lemma} completes this case.

 If $s=t$,  Claim~\ref{part-plus-cell} gives $s+t=p-1$, so $p=2s+1$. Moreover, each 
part of $R$ contains exactly $1$ cell. From here it is straightforward to count the number of $K_q$ copies in a maximal $R$. Applying Lemma~\ref{key-lemma} completes this case.
\end{proof}

The second frequently applied case is when $s=q$ and we use the following lemma.

\begin{lemma}\label{Bisq}
If $s=q$, then 
\[
\mathfrak{R}_q(p) = \left( \frac{1}{q} \right)^q \left(\frac12\right)^{\binom{q}{2} - e(T(q,p-q-1))},
\]
where $T(q,p-q-1)$ is the Tur\'an graph on $q$ vertices with $p-q-1$ classes.
Alternatively, if $s=q$, then
\[
\N_q(G) = (1+o(1)) \max \left \{ \N_q(H) \mid H \in \mathcal{G}(n;q,p-q-1) \right \}.
\] 
\end{lemma}

\begin{proof}
    Claim~\ref{part-plus-cell} gives $s+t = p-1$, so $t = p-q-1$.
As there are exactly $q$ cells in $R$, each copy of $K_q$ has exactly the same weight. Therefore, to maximize the number of $K_q$ copies, all $q$ cells should be of equal size. Each clique has weight $\left(\frac{1}{2}\right)^{\binom{s_1}{2}+\binom{s_2}{2}+\ldots+\binom{s_t}{2}}$. This weight is maximized when the exponent is minimized. Translating to the problem of minimizing non-edges in a $(p-q-1)$-partite graph on $q$ vertices gives weight $\left(\frac12\right)^{\binom{q}{2} - e(T(q,p-q-1))}$ for cliques in maximal $R$.
Applying Lemma~\ref{key-lemma} completes the proof.
\end{proof}

There remain several cases ($q=5$ and $p=10,11$) in the following subsections that do not use Lemma~\ref{p-large} or~\ref{Bisq} but involve similar optimizations. We will leave the details of these straightforward maximizations to the reader.

\subsection{\texorpdfstring{The cases $q=2,3$}{text}}

 Now we can give a new proof of Theorems~\ref{edge-thm} and~\ref{triangle-thm} (i.e., the cases $q=2,3$). 
 This will give a general outline for how the method will proceed for $q=4,5$.
 
 Suppose $2 \leq q \leq 3$ and $p \geq 2q$. 
 First, suppose that $s-t \geq 2$. Then $2q-1\leq p-1=s+t\leq 2s-2$ implies $s>q$. Then by Claims~\ref{cell-lemma1} and~\ref{cell-lemma2}  we can transform $R$ such that each part of $R$ contains at most two cells without decreasing $\N_q(R)$. 
Suppose that there are two parts $i$ and $j$ of $R$ that  each contains two cells.
By Lemma~\ref{repartition-two-two}, there is a transformation that locally increases the number of edges and the number of triangles while leaving the number of vertices unchanged.

Now, every clique on at most three vertices in $R$ uses at most three vertices among those we have transformed, call these vertices $X$. Each vertex not in $X$ is connected by an edge of weight $1$ to a vertex of $X$. It follows that if we increase edges and triangles locally, without losing vertices, the global number of edges and triangles must also increase.
 We may repeatedly apply such a transformation on pairs of parts containing two cells without reducing $\N_q(R)$. So we may assume at most one part in $R$ contains two cells, i.e., $s - t \leq 1$.
 Invoking Lemma~\ref{p-large} and evaluating the four cases of $q=2$ or $q=3$, $p=2s+1$ or $p=2s$ yields the two theorems with the maximum achieved by a graph in $\mathcal{G}(n;\lceil (p-1)/2 \rceil,\lfloor (p-1)/2 \rfloor)$. We remark that for $q=2$, $p=2s\geq 4$, the maximum value of $\frac{3s-5}{6s-4}$ can be explicitly computed and is achieved at $x=\frac{4}{3(s-1)}$
 in the optimization of $\mathfrak{R}_q(2s)$.

\subsection{\texorpdfstring{The case $q=4$}{text}}

For $6 \leq p \leq 7$, the theorem follows from part (a) of Theorem~\ref{q-plus}. 
Suppose that $s - t \geq 2$. Then $7 \leq p-1 =s+t \leq 2s-2$ implies $s \geq 5$.
By Corollary~\ref{cell-cor} we may assume that each part of $R$ contains at most three cells. Suppose that 
some part $i$ contains three cells, each of cardinality $2x$.
We transform part $i$ by partitioning it into two parts each of a single cell of size $3x$.
The number of parts increases by $1$ and the number of cells decreases by $1$, so $R$ remains $p$-skeleton-free.
Let $W$ be the subgraph of $R$ induced on all vertices not in part $i$, the graph $W$ and the edges between it and part $i$ are unchanged by the transformation.
        As $s \geq 5$, there are at least two cells in $W$ and by Claim~\ref{relative-cell}, each is of size at least $2x$, the size of a cell in part $i$. Therefore, $|V(W)| \geq 4x$ and the total edge weight in $W$ satisfies $\N_2(W) \geq \frac{1}{2}(2x)(|V(W)|-2x)$. 
    Thus, the change in the number  of $K_4$ is
    \begin{align*}
    &\left((3x)^2-3\frac{1}{2}(2x)^2\right)\N_2(W) -\frac{1}{8}(2x)^3|V(W)|  = 3x^2\N_2(W) - x^3|V(W)| \\
    & \geq 2x^3|V(W)| - 6x^4 
 \geq 8x^4 - 6x^4 
     > 0.
        \end{align*}
        
        Therefore, this transformation does not decrease $\N_4(R)$ and so  we may assume that
        each part of $R$ contains at most two cells.
        Suppose that $R$ has two parts that contain two cells. By Claim~\ref{relative-cell} we may suppose that each of these four cells has cardinality $3x$.
By Lemma~\ref{repartition-two-two}, there is a transformation that locally increases the number of edges by $3x^2$, and the number of triangles by $10x^3$ and decreases the number of $K_4$ copies by $\frac{81}{4}x^4$, while leaving the number of vertices unchanged.         
       Again, as $s \geq 5$, there is at least one cell in $W$. Therefore, $|V(W)| \geq 3x$.
The number  of $K_4$ that use at most one vertex from part $i$ remains unchanged. The number of $K_4$ copies that use at least two vertices in $i$ increases by
\[
3x^2 \N_2(W) + 10x^3 |V(W)| - \frac{81}{4}x^4 \geq 30x^4 - \frac{81}{4}x^4  >0.
\]
Therefore, this transformation does not decrease $\N_4(R)$ and so we may assume that
        there is at most one part with two cells, i.e., $s-t \leq 1$. Invoking Lemma~\ref{p-large} and evaluating the two cases $p=2s+1$ or $p=2s$ yields the theorem
with the maximum achieved by a graph in $\mathcal{G}(n;\lceil (p-1)/2 \rceil,\lfloor (p-1)/2 \rfloor)$.

\subsection{\texorpdfstring{The case $q=5$}{text}}

The cases $7 \leq p \leq 8$ follow from part (a) of Theorem~\ref{q-plus}, so let $p \geq 9$. 
We may assume that each part contains at most three cells: if $s=5$, then $t=p-q-1\geq 3$, and it is easy to see that no part contains more than $s-t+1\leq 3$ cells. If $s>5$, then Claims~\ref{cell-lemma1} and~\ref{cell-lemma2} imply that each part of contains at most three cells.

The next lemma also holds when counting copies of $K_q$ in general, but we will use it only for $q=5$  here.

\begin{lemma}\label{no31}
We may suppose the reduced graph $R$ does not contain both a part with exactly $3$ cells and a part with exactly $1$ cell.
\end{lemma}

\begin{proof}
    Suppose that part $i$ contains exactly three cells of (normalized) cardinality $x$ and $j$ contains exactly one cell of cardinality $1-3x$. By Claim~\ref{relative-cell} we can assume $1-3x \geq x$ which implies $x \leq \frac{1}{4}$. 
    We will transform $R$ by replacing the vertices of parts $i$ and $j$ with two parts each containing two cells of cardinality $\frac{1}{4}$. As the total number of parts and cells remains the same, $R$ is still $p$-skeleton-free.

    Let us compute the change in the (normalized) number of edges, triangles and $K_4$ copies  spanned by the vertices in classes $i$ and $j$ under this transformation. Initially, there are $3x(1-3x)+\frac{3}{2}x^2$ edges, $\frac{1}{8}x^3+\frac{3}{2}x^2(1-3x)$ triangles,  and $\frac{1}{8}x^3(1-3x)$ copies of $K_4$. Thus, there are at most $\frac{3}{10}$ edges, $\frac{32}{1225}$ triangles, and $\frac{1}{2048}$ copies of $K_4$. After the transformation, there are $\frac{5}{16}$ edges, $\frac{1}{32}$ triangles, and $\frac{1}{1024}$ copies of $K_4$. Therefore, all three quantities increase under the transformation.
    As every  $K_5$ in $R$ intersects the vertices of parts $i$ and $j$ in at most $4$ vertices, the number of $K_5$ copies  increases.  
\end{proof}

We now distinguish four cases based on the value of $p$, using Lemma~\ref{no31} and Claim~\ref{cell-lemma2} along with $s+t=p-1$ (by Claim~\ref{part-plus-cell}) to determine possible cell profiles. For simplicity we normalize $|V(R)|$ to $1$.

\smallskip

{\bf Case 1:} $p = 9$. 

\smallskip

In this case the possible cell profiles of $R$ are
 $(3,3)$ or $(2,2,1)$. If $(3,3)$ is the cell profile, then all parts contain the same number of cells, thus all cells must be of the same size. Then $R$ has $\frac{6}{16}\left(\frac{1}{6}\right)^5$ copies of $K_5$. If $(2,2,1)$ is the cell profile, then $s=5$, and so (as in the proof of Lemma~\ref{Bisq}) all cells are the same size. Then $R$ has $\frac{1}{4}\left(\frac{1}{5}\right)^5$ copies of $K_5$, which is larger than the number for the cell profile $(3,3)$.
Invoking Lemma~\ref{Bisq} completes this case with the maximum achieved by a graph in $\mathcal{G}(n;5,3)$.

\smallskip

{\bf Case 2:} $p = 10$. 

\smallskip
In this case the possible cell profiles are
$(2,2,2)$ or $(2,1,1,1)$. In the first profile, all parts contain the same number of cells, and in the latter there are $5$ cells. The argument mirrors Case 1 and we get $K_5$ counts of $\frac{6}{4}\left(\frac{1}{6}\right)^5$ and $\frac{1}{2}\left(\frac{1}{5}\right)^5$ respectively. One can check to see the former expression, which corresponds to the former profile, is larger (while the Conjecture~\ref{their-conjecture} would support the latter).
Invoking Lemma~\ref{key-lemma} completes this case with the maximum achieved by a graph in  $\mathcal{G}(n;6,3)$.

\smallskip

{\bf Case 3:} $p = 11$. 

\smallskip
In this case the possible cell profiles are
 $(3,2,2)$, $(2,2,1,1)$, or $(1,1,1,1,1)$. 
If the profile is $(1,1,1,1,1)$, then Claim~\ref{relative-cell} implies that cells are the same size so there are $\left(\frac{1}{5}\right)^5$ copies of $K_5$.
If the profile is $(3,2,2)$ or $(2,2,1,1)$, then cells are potentially of different sizes, so there is an obvious (univariate) optimization problem to be solved.
One can check that the maximum across all three profiles is achieved by $(2,2,1,1)$.
Invoking Lemma~\ref{key-lemma} completes this case with the maximum achieved by a graph in $\mathcal{G}(n;6,4)$.

\smallskip

{\bf Case 4:} $p = 12$. 

\smallskip
In this case the possible cell profiles are $(2,1,1,1,1)$ or $(2,2,2,1)$. Solving the optimization problem for cell sizes in each profile shows that the optimum is achieved by $(2,1,1,1,1)$. Invoking Lemma~\ref{p-large} completes this case with the maximum achieved by a graph in  
$\mathcal{G}(n;6,5)$.
\smallskip

{\bf Case 5:} $p = 13$. 

\smallskip

In this case the possible cell profiles are $(1,1,1,1,1,1)$ or $(2,2,1,1,1)$. In the first case, all parts contain the same number of cells, so cell sizes must be balanced, and $R$ has $\left(\frac{1}{6}\right)^5$ copies of $K_5$. One may solve the optimization problem for the latter profile, finding the optimum is still achieved by the former profile $(1,1,1,1,1,1)$. Invoking Lemma~\ref{p-large} completes this case with the maximum achieved by a graph in $\mathcal{G}(n;6,6)$.

\smallskip

{\bf Case 6:} $p \geq 14$. 

\smallskip

Suppose that $s-t\geq 2$. Then $2s-2 \geq s+t\geq 13$. So $s\geq 8$, and Claim~\ref{cell-lemma2} immediately implies that no part contains three cells.
Therefore, there are parts $i$ and $j$ both containing exactly two cells. By Claim~\ref{relative-cell} we may suppose each of these four cells has cardinality $3x$. By Lemma~\ref{repartition-two-two}, there is a transformation that locally increases the number of edges by $3x^2$ and increases the number of triangles by $10x^3$ and decreases the number of $K_4$ copies by $\frac{81}{4}x^4$ while leaving the number of vertices unchanged.

Let $W$ be the subgraph of $R$ induced on all vertices not in part $i$ or $j$. Clearly, $|V(W)|=1-12x$. By Claim~\ref{relative-cell}, every cell in $W$ has cardinality at least $3x$. Therefore, the number of triangles $\N_3(W)$ in $W$ satisfies
\[
\N_3(W) \geq \frac{1}{2}\binom{s-4}{3}(3x)^3,
\]
as each triangle contains at most one edge of weight $1/2$.
 Similarly, the number of edges $\N_2(W)$ in $W$ satisfies
 \[
 \N_2(W) \geq \left(\binom{s-4}{2}-\frac{1}{2}\left \lfloor \frac{s-4}{2}\right \rfloor\right) (3x)^2,
 \]
 as there are at most $\lfloor \frac{s-4}{2}\rfloor$ parts in $W$ with two cells. 

If every part in $W$ contains exactly two cells, then each cell has cardinality $3x = \frac{1}{s}$. Otherwise, there is a part in $W$ with exactly one cell. Locally investigating a part with two cells of cardinality $3x$ and another part with a single cell of cardinality $y\geq 3x$, one can check via an easy optimization that edge and triangle counts both can increased by rebalancing whenever $3x<\frac{2}{3}y$. Now, each cell is of one of two sizes: $3x$, or $y\leq \frac{9}{2}x$. Let $k$ be the number cells in parts with two cells and $s-k$ the number of remaining cells, each of cardinality $y$. Given our prior normalization we have $3xk+y(s-k)=1$. Substituting  $y \leq \frac{9}{2}x$ and solving for $x$ yields $x\geq \frac{2}{9s-3k}$. As parts $i$ and $j$ each have two cells, $k\geq 4$, so $x\geq \frac{2}{9s-12}$.

Putting this all together, the value of $\N_5(R)$ increases by at least
\begin{align*}
 3x^2\N_3(W) & + 10x^3\N_2(W) - \frac{81}{4} x^4 |V(W)| \\
 & \geq 3x^2 \frac{1}{2}\binom{s-4}{3}(3x)^3 + 10x^3 \left(\binom{s-4}{2}-\frac{1}{2}\left \lfloor \frac{s-4}{2}\right \rfloor\right) (3x)^2 - \frac{81}{4} x^4 (1-12x) \\
 &= x^4\left[ x\left(\frac{81}{2} \binom{s-4}{3} +90\binom{s-4}{2} -45 \left \lfloor \frac{s-4}{2} \right\rfloor +243 \right) - \frac{81}{4} \right] \\
 & \geq x^4\left[ \frac{2}{9s-12}\left(\frac{81}{2} \binom{s-4}{3} +90\binom{s-4}{2} -45 \left \lfloor \frac{s-4}{2} \right\rfloor +243 \right) - \frac{81}{4} \right].
\end{align*}

This value is positive for $s \geq 8$. Indeed, evaluating at $s=8$ yields a positive value, and elementary calculus shows that the resulting function, which is quadratic in $s$, has positive derivative for all $s\geq 8$.
 Therefore, this transformation does not decrease $\N_5(R)$ and so we may assume that
        there is at most one part with two cells, i.e., $s-t \leq 1$. 
Now, invoking Lemma~\ref{p-large} and evaluating the two cases $p=2s+1$ or $p=2s$ yields the theorem with the maximum achieved by a graph in $\mathcal{G}(n;\lceil (p-1)/2 \rceil,\lfloor (p-1)/2 \rfloor)$.

\subsection{Proof of Theorem~\ref{q-plus}}

In this subsection we prove Theorem~\ref{q-plus} which addresses the case when $q+2 \leq p \leq q+4$. Note that  the proof of Theorems~\ref{edge-thm}, \ref{triangle-thm}, \ref{k4-thm}, and \ref{k5-thm} only use part (a) so we may apply them in the proof of part (b).

 For simplicity we normalize the number of vertices in $R$ to $1$.
We begin by proving part (a). 
 If $p=q+2$, then Claim~\ref{part-plus-cell} implies $s+t=p-1=q+1$.  Therefore, as $s \geq q$ and $t\geq 1$, we have $t=1$ and $s=q$.
Lemma~\ref{Bisq} yields the desired result with the maximum achieved by a graph in $\mathcal{G}(n;q,1)$.

    If $p=q+3$, then Claim~\ref{part-plus-cell} implies $s=q+1$ or $s=q$. In the former case, $t=1$, i.e., there is only one part, so the cells are of equal size $\frac{1}{q+1}$. Thus, the (normalized) number of $K_q$ copies is $(q+1)\left(\frac{1}{q+1}\right)^q\left(\frac{1}{2}\right)^{\binom{q}{2}}$.
    Supposing the latter case, $s=q$ and $t=2$, Lemma~\ref{Bisq} implies that the (normalized) number of $K_q$ copies is $\left(\frac{1}{q}\right)^q\left(\frac{1}{2}\right)^{\binom{\lfloor q/2\rfloor}{2}+ \binom{\lceil q/2 \rceil}{2}}= \left(\frac{1}{q}\right)^q\left(\frac{1}{2}\right)^{\binom{q}{2}}2^{\lfloor q/2\rfloor\lceil q/2 \rceil}$. 
    It is straightforward to check that there are more copies of $K_q$ in the case $s=q \geq 2$. So Lemma~\ref{Bisq} yields the desired result, with the maximum achieved by a graph in $\mathcal{G}(n;q,2)$.

\smallskip

It remains to prove part (b). Suppose $p=q+4$. For $q=3,4,5$, 
the statement follows from Theorems~\ref{triangle-thm}, \ref{k4-thm}, and \ref{k5-thm} so assume $q \geq 6$.
By Claim~\ref{part-plus-cell} we have $s+t=p-1=q+3$. Therefore, $s \geq q$ implies $1 \leq t \leq 3$. If $t=1$, then $s=q+2$, which violates our assumption on the structure of $R$ from Claim~\ref{cell-lemma1}. Then assume that $t=2$. This implies $s=q+1$, so some part contains at least $\lceil \frac{q+1}{2} \rceil$ cells. When $s=q+1$, then (\ref{k-bound}) in Claim~\ref{cell-lemma2} implies that the maximum number $k$ of cells in a part satisfies $2^{k-2}\left(\frac{k}{k-1}\right)^{k-1} \leq q+1$ or $k \leq 2$. When $q \geq 6$, and $k \geq \lceil \frac{q+1}{2} \rceil$ this is impossible. 
     Therefore, $t=3$ and $s=q$ and Lemma~\ref{Bisq} yields the desired result, with the maximum achieved by a graph in $\mathcal{G}(n;q,3)$ for $q \geq 3$.

\subsection{\texorpdfstring{General $q$}{text}}

Below we provide a general result showing that Conjecture~\ref{their-conjecture} eventually holds. Note that we make limited effort to improve the constant $5$ below: a new approach or more sophisticated analysis of $\N_{q-2}(R)$ is needed for this method to  yield bounds close to what one might expect are best-possible.

\begin{proof}[Proof of Theorem~\ref{general-q}.]

  Let $q \geq 5$ and $p \geq 5q$. Suppose that $s-t \geq t$.
 As $s+t =p-1$ and $s\geq t$ we have $s\geq \lceil\frac{p-1}{2}\rceil\geq 2.4q$. By Claims~\ref{cell-lemma2} and \ref{cell-lemma1} we may assume that each part of $R$ contains at most two cells.
As such, $R$ must contain two parts $i$ and $j$ each with two cells. Let each of these four cells have $3x$ vertices. By Lemma~\ref{repartition-two-two} there is a transformation that locally increases the number of edges by $3x^2$, the number of triangles by $10x^3$, and decreases the number of $K_4$ copies by $\frac{81}{4}x^4$. 

    Now, let $W$ be the subgraph of $R$ induced on the vertices outside of part $i$ and $j$. As in Claim~\ref{cell-lemma2}, we can double count pairs $(K,u)$ where $K$ is a $(q-3)$-clique in $W$ and $u$ a vertex in $K$. We see that $u$ is in a part with at most two cells, so it has at most one neighbor $v$ in $K$ for which $w(uv)={1/2}$, yielding

    \[(q-3)\N_{q-3}(W)\geq \N_{q-4}(W)(s-q)3x\left(\frac{1}{2}\right).\]

    It is clear that after the local transformation given by Lemma~\ref{repartition-two-two}, $R$ contains at least as many copies of $K_q$ using $q-2$ or more vertices of $W$. We see that the transformation increases the number  of $K_q$ using $q-3$ or $q-4$ vertices from $W$ by:
    \[
    10x^3\N_{q-3}(W)-\frac{81}{4}x^4\N_{q-4}(W),
    \]
    which, after substituting $3x\N_{q-4}(W)$ for its upper bound of $2\frac{q-3}{s-q}\N_{q-3}(W)$, must be positive when $10-(\frac{27}{2})(\frac{q-3}{s-q})$ is positive. It is straightforward to check that this is the case when $s\geq 2.4q$. Therefore, this transformation does not decrease $\N_q(R)$ and so we may assume that there is at most one part with two cells, i.e., $s-t \leq 1$. 
Invoking Lemma~\ref{p-large} and evaluating the two cases $p=2s+1$ or $p=2s$ yields the theorem. 
\end{proof}

\section{Concluding remarks}

We conclude with two more general bounds that may be of use in extending some of the results in this manuscript.
Broadly speaking, when combined with Lemma~\ref{Bisq}, the following proposition implies that if $q$ is large and $p<q+\frac{q}{\log q}+1$ (so $p$ is quite small relative to $q$), a graph in $\mathcal{G}(n;q,p-q-1)$ is asymptotically extremal and Conjecture~\ref{their-conjecture} holds in this case.

\begin{prop}
        For all $0<c<1$ and $q$  large enough, if $p\geq q+c\left(\frac{q}{\log q}\right)+1$, then $t\geq c\left(\frac{q}{\log q}\right)$ where $t$ is the number of parts in $R$.    
\end{prop}

\begin{proof}
    Suppose $t<c\left(\frac{q}{\log q}\right)$. Then the number of cells $s$ is at least $q+1$ and there exists a part containing at least $\frac{1}{c}\log q$ cells. For any $0<c<1,$ we have $2^{\frac{1}{c}\log q}>4q+4>2$ for $q$ large enough, which violates Claim~\ref{cell-lemma2}.
\end{proof}

In general, when $p$ is small relative to $q$ (i.e., $p<5q$), the values of $s$ and $t$ for which $\be(n;s,t)$ asymptotically achieves $\rt(n,K_q,K_p,o(n))$ are unknown. We do know that $s+t=p-1$, but as we discuss in Section~\ref{section-counter}, there are many instances in which $s\neq \max\{q,\lceil (p-1)/2\rceil\}$ and Conjecture~\ref{their-conjecture} does not hold. By examining Claim~\ref{cell-lemma2}, one may observe that no part should contain more than $O(\log q)$ many cells, this fact is exploited above. What comes below expands on this idea. Indeed, when $p$ is small relative to $q$ (say, $p<2q$), $s=p-2$, $t=1$ is a feasible output of the symmetrization in Theorem~\ref{w-zykov}. However, we have shown this violates Claim~\ref{cell-lemma1}. We can in general apply Claim~\ref{cell-lemma2} and bound $s$ away from $p$ by a linear factor.

\begin{prop}
        For all  $\varepsilon>0$, there exists $\delta$ satisfying $0<\delta< \frac{1}{2}\varepsilon$ and  $2^{1/\delta}>\frac{4}{\varepsilon}+\frac{2}{\delta}$,  such that whenever $p-1\geq (1+\varepsilon)q$, and $q$ is large enough, then $t\geq \delta q$ where $t$ is the number of parts in $R$.  
\end{prop}
\begin{proof}
    Suppose not, so $t<\delta q$. Then, as $p-1=s+t$, we have that $s$, the number of cells in $R$, satisfies
    $s> (1+\varepsilon -\delta)q\geq q$. 
    Let $k$ be the maximum number of cells in a part of $R$.
By averaging, there is a part containing at least $\frac{s}{t} > \frac{q}{\delta q} = \frac{1}{\delta}>2$ cells. 
Given the definition of $\delta$, one can check that
\[
2^{k-2} \left(\frac{k}{k-1}\right)^{k-1} - k > 2^{1/\delta -1} - \frac{1}{\delta} > \frac{2}{\varepsilon} > \frac{q-k+1}{s-q}.
\]
which violates (\ref{k-bound}) in Claim~\ref{cell-lemma2}.
\end{proof}

\bibliographystyle{hplain}

\begin{thebibliography}{10}
	
	\bibitem{AlSh}
	Noga Alon and Clara Shikhelman.
	\newblock Many {$T$} copies in {$H$}-free graphs.
	\newblock {\em J. Combin. Theory Ser. B}, 121:146--172, 2016.
	
	\bibitem{andras}
	B.~Andr\'{a}sfai.
	\newblock Graphentheoretische {E}xtremalprobleme.
	\newblock {\em Acta Math. Acad. Sci. Hungar.}, 15:413--438, 1964.
	
	\bibitem{BCMM}
	J\'{o}zsef Balogh, Ce~Chen, Grace McCourt, and Cassie Murley.
	\newblock Ramsey-{T}ur\'{a}n problems with small independence numbers.
	\newblock {\em European J. Combin.}, 118:Paper No. 103872, 15, 2024.
	
	\bibitem{BaloghLenz}
	J\'{o}zsef Balogh and John Lenz.
	\newblock Some exact {R}amsey-{T}ur\'{a}n numbers.
	\newblock {\em Bull. Lond. Math. Soc.}, 44(6):1251--1258, 2012.
	
	\bibitem{balogh2017problems}
	J\'{o}zsef Balogh, Hong Liu, and Maryam Sharifzadeh.
	\newblock On two problems in {R}amsey-{T}ur\'{a}n theory.
	\newblock {\em SIAM J. Discrete Math.}, 31(3):1848--1866, 2017.
	
	\bibitem{balogh2023weighted}
	József Balogh, Domagoj Bradač, and Bernard Lidický.
	\newblock Weighted {T}ur\'an theorems with applications to {R}amsey-{T}ur\'an
	type of problems, 2023, arXiv:2302.07859.
	
	\bibitem{bo-book}
	B\'{e}la Bollob\'{a}s.
	\newblock {\em Modern graph theory}, volume 184 of {\em Graduate Texts in
		Mathematics}.
	\newblock Springer-Verlag, New York, 1998.
	
	\bibitem{BE-const}
	B\'{e}la Bollob\'{a}s and Paul Erd\H{o}s.
	\newblock On a {R}amsey-{T}ur\'{a}n type problem.
	\newblock {\em J. Combinatorial Theory Ser. B}, 21(2):166--168, 1976.
	
	\bibitem{BoEr}
	B\'{e}la Bollob\'{a}s and Paul Erd\H{o}s.
	\newblock On a {R}amsey-{T}ur\'{a}n type problem.
	\newblock {\em J. Combinatorial Theory Ser. B}, 21(2):166--168, 1976.
	
	\bibitem{EHSSz}
	P.~Erd\H{o}s, A.~Hajnal, Vera~T. S\'{o}s, and E.~Szemer\'{e}di.
	\newblock More results on {R}amsey-{T}ur\'{a}n type problems.
	\newblock {\em Combinatorica}, 3(1):69--81, 1983.
	
	\bibitem{ErSi}
	P.~Erd\H{o}s and M.~Simonovits.
	\newblock A limit theorem in graph theory.
	\newblock {\em Studia Sci. Math. Hungar.}, 1:51--57, 1966.
	
	\bibitem{ErSoRT}
	P.~Erd\H{o}s and Vera~T. S\'{o}s.
	\newblock Some remarks on {R}amsey's and {T}ur\'{a}n's theorem.
	\newblock In {\em Combinatorial theory and its applications, {I}-{III} ({P}roc.
		{C}olloq., {B}alatonf\"{u}red, 1969)}, volume~4 of {\em Colloq. Math. Soc.
		J\'{a}nos Bolyai}, pages 395--404. North-Holland, Amsterdam-London, 1970.
	
	\bibitem{ErSt}
	P.~Erd\H{o}s and A.~H. Stone.
	\newblock On the structure of linear graphs.
	\newblock {\em Bull. Amer. Math. Soc.}, 52:1087--1091, 1946.
	
	\bibitem{FuSi}
	Zolt\'{a}n F\"{u}redi and Mikl\'{o}s Simonovits.
	\newblock The history of degenerate (bipartite) extremal graph problems.
	\newblock In {\em Erd\"{o}s centennial}, volume~25 of {\em Bolyai Soc. Math.
		Stud.}, pages 169--264. J\'{a}nos Bolyai Math. Soc., Budapest, 2013.
	
	\bibitem{gao2024generalized}
	Jun Gao, Suyun Jiang, Hong Liu, and Maya Sankar.
	\newblock Generalized {R}amsey--{T}ur\'an density for cliques, 2024,
	arXiv:2403.12919.
	
	\bibitem{reg-survey}
	J.~Koml\'{o}s and M.~Simonovits.
	\newblock Szemer\'{e}di's regularity lemma and its applications in graph
	theory.
	\newblock In {\em Combinatorics, {P}aul {E}rd\H{o}s is eighty, {V}ol. 2
		({K}eszthely, 1993)}, volume~2 of {\em Bolyai Soc. Math. Stud.}, pages
	295--352. J\'{a}nos Bolyai Math. Soc., Budapest, 1996.
	
	\bibitem{LiPf}
	Bernard Lidick\'{y} and Florian Pfender.
	\newblock Pentagons in triangle-free graphs.
	\newblock {\em European J. Combin.}, 74:85--89, 2018.
	
	\bibitem{liu2021geometric}
	Hong Liu, Christian Reiher, Maryam Sharifzadeh, and Katherine Staden.
	\newblock Geometric constructions for {R}amsey-{T}ur\'an theory, 2021,
	arXiv:2103.10423.
	
	\bibitem{rt-survey}
	Mikl\'{o}s Simonovits and Vera~T. S\'{o}s.
	\newblock Ramsey-{T}ur\'{a}n theory.
	\newblock volume 229, pages 293--340. 2001.
	\newblock Combinatorics, graph theory, algorithms and applications.
	
	\bibitem{Sz4}
	Endre Szemer\'{e}di.
	\newblock On graphs containing no complete subgraph with {$4$} vertices.
	\newblock {\em Mat. Lapok}, 23:113--116, 1972.
	
	\bibitem{SZRL}
	Endre Szemer\'{e}di.
	\newblock Regular partitions of graphs.
	\newblock In {\em Probl\`emes combinatoires et th\'{e}orie des graphes
		({C}olloq. {I}nternat. {CNRS}, {U}niv. {O}rsay, {O}rsay, 1976)}, volume 260
	of {\em Colloq. Internat. CNRS}, pages 399--401. CNRS, Paris, 1978.
	
	\bibitem{zykov}
	A.~A. Zykov.
	\newblock On some properties of linear complexes.
	\newblock {\em Mat. Sbornik N.S.}, 24(66):163--188, 1949.
	
\end{thebibliography}

\end{document}